\documentclass[12pt,english]{amsart}
\usepackage[T1]{fontenc}
\usepackage[latin9]{inputenc}  

\usepackage{babel}
\usepackage{marginnote} 
\usepackage[margin=1in]{geometry}
\usepackage{fancyhdr}
\usepackage{cite,hyperref}
\usepackage{amsmath, amsthm, amssymb}
\usepackage{bm} 
\usepackage[table]{xcolor}
\usepackage{wasysym,esint}
\usepackage{enumerate}

\usepackage{tikz}
	\usetikzlibrary{positioning}
    \usetikzlibrary{decorations.markings}
    \usetikzlibrary{hobby}
    \usetikzlibrary{matrix, arrows}
    \usetikzlibrary{decorations.pathreplacing}
\usepackage{pgfplots}
	\pgfplotsset{compat=newest}
	\usepgfplotslibrary{fillbetween}

\hypersetup{
    colorlinks = true,          
    linkcolor=teal,             
    urlcolor=blue,              
    citecolor=cyan,             
    filecolor=magenta,          
}

\colorlet{bl}{blue!70}
\colorlet{or}{orange!70}

\newtheorem{theorem}{Theorem}[section]
\newtheorem{corollary}[theorem]{Corollary}
\newtheorem{lemma}[theorem]{Lemma}
\newtheorem{proposition}[theorem]{Proposition}
\newtheorem{definition}{Definition}[section]
\newtheorem{remark}{Remark}[section]
\newtheorem{conjecture}{Conjecture}[section]

\newtheoremstyle{example}{3pt}{3pt}{}{0pt}{\bfseries}{. }{ }{}
\theoremstyle{example}
\newtheorem{example}{Example}
\newtheorem{note}{Note}[section]

\let\emptyset\varnothing    

\title{Accessibility of the Boundary of the Thurston Set}
\date{\today}

\begin{author}[S. Silvestri]{Stefano Silvestri}
\address{ %
Butler University Department of Mathematics, Statistics, and Actuarial Science\\
4600 Sunset Avenue\\
Indianapolis, Indiana 46208\\
 United States }
\email{ssilvestri@butler.edu}
\end{author}
\begin{author}[R.~A. P\'erez]{Rodrigo A.~P\'erez}
\address{ %
IUPUI Department of Mathematical Sciences\\
402 North Blackford Street\\
Indianapolis, Indiana 46202\\
 United States }
\email{rperez@iupui.edu}
\end{author}

\begin{document}

\begin{abstract}
    Consider two objects associated to the Iterated Function System (IFS) $\{1+\lambda z,-1+\lambda z\}$: the locus $\mathcal{M}$ of parameters $\lambda\in\mathbb{D}\setminus\{0\}$  for which the corresponding attractor is connected; and the locus $\mathcal{M}_0$ of parameters for which the related attractor contains $0$. The set $\mathcal{M}$ can also be characterized as the locus of parameters for which the attractor of the IFS $\{1+\lambda z, \lambda z, -1+\lambda z\}$ contains $\lambda^{-1}$. Exploiting the asymptotic similarity of $\mathcal{M}$ and $\mathcal{M}_0$ with the respective associated attractors, we give sufficient conditions on $\lambda\in\partial\mathcal{M}$ or $\partial\mathcal{M}_0$ to guarantee it is path accessible from the complement $\mathbb{D}\setminus\mathcal{M}$.
\end{abstract}

\maketitle

\section{Introduction}
For any $\lambda$ in the punctured disk $\mathbb{D}^*:=\mathbb{D}\setminus\{0\}$ the maps
\[
    \mathfrak{s}_-(z):=-1+\lambda z, \qquad \mathfrak{s}_{\mathtt{O}}(z):=\lambda z, \qquad \mathfrak{s}_+(z):=1+\lambda z
\]
are contraction similarities in all of $\mathbb{C}$. We associate to $\lambda$ the compact sets 
\[\mathsf{A}_\lambda:=\left\{\sum_{n=0}^\infty a_n\lambda^n ~\bigg|~a_n\in\{-1,+1\}\right\}\text{ and }\widetilde{\mathsf{A}}_\lambda:=\left\{\sum_{n=0}^\infty a_n\lambda^n ~\bigg|~a_n\in\{-1,0,+1\}\right\}.\]
These sets are the attractors of the iterated function systems or IFS $\{\mathfrak{s}_-,\mathfrak{s}_+\}$ and $\{\mathfrak{s}_-,\mathfrak{s}_{\mathtt{O}},\mathfrak{s}_+\}$, respectively; that is, the unique non-empty compact sets satisfying \[\mathsf{A}_\lambda=\mathfrak{s}_-(\mathsf{A}_\lambda)\cup \mathfrak{s}_+(\mathsf{A}_\lambda)~\text{ and }~\widetilde{\mathsf{A}}_\lambda=\mathfrak{s}_-(\widetilde{\mathsf{A}}_\lambda)\cup\mathfrak{s}_{\mathtt{O}}(\widetilde{\mathsf{A}}_\lambda)\cup\mathfrak{s}_+(\widetilde{\mathsf{A}}_\lambda).\]

Both attractors are symmetric about $0$, and clearly $\mathsf{A}_\lambda\subseteq\widetilde{\mathsf{A}}_\lambda$. It is advantageous to study them together because, in contrast to what happens with the Mandelbrot set, the parameter regions 
\[
\mathcal{M}:=\{\lambda\in\mathbb{D}~|~\mathsf{A}_\lambda \mbox{ is connected} \}~\mbox{ and }~\mathcal{M}_0:=\{\lambda\in\mathbb{D}~|~0\in\mathsf{A}_\lambda \},
\]
associated to the $2$-map IFS $\{\mathfrak{s}_-,\mathfrak{s}_+\}$, do not coincide (see Figure~\ref{fig:MandM0}). However, there is an interesting connection between the two sets, because 
\[
  \mathcal{M} = \big\{\lambda\in\mathbb{D}~|~0\in\mathfrak{s}_-(\widetilde{\mathsf{A}}_\lambda)\cap\mathfrak{s}_{\mathtt{O}}(\widetilde{\mathsf{A}}_\lambda)\cap\mathfrak{s}_+(\widetilde{\mathsf{A}}_\lambda)\big\} \text{ while }
  \mathcal{M}_0 = \big\{\lambda\in\mathbb{D}~|~0\in\mathfrak{s}_-(\mathsf{A}_\lambda)\cap\mathfrak{s}_+(\mathsf{A}_\lambda)\big\}
\]
(cf. Note~\ref{note:2}). Following \cite{BDLW}, we refer to $\mathcal{M}_0$ as the {\em Thurston set}.
\begin{figure}
\centering
    \includegraphics[scale=0.282]{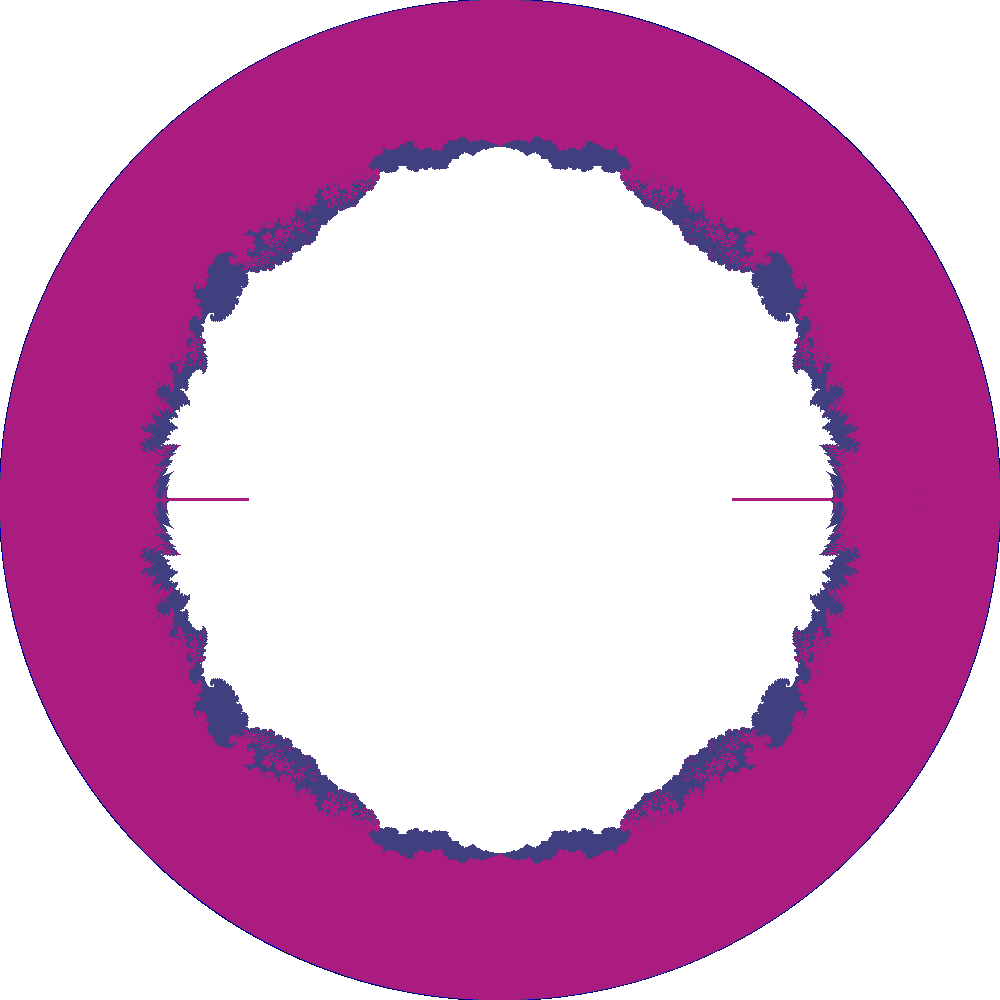}
\caption{The connectivity locus $\mathcal{M}$ and its subset $\mathcal{M}_0$, in blue and fuchsia respectively. The two spikes on the real line start at $\pm1/2$.}
\label{fig:MandM0}
\end{figure}

Interest in these IFS families spiked recently, after Tiozzo \cite{Ti} (inspired by a conjecture of Thurston \cite{Th}) proved that $\mathcal{M}_0$ equals the closure of the set of Galois conjugates of entropies of superattracting real quadratic polynomials. Both $\mathcal{M}$ and $\mathcal{M}_0$ were first introduced by Barnsley and Harrington \cite{BarHar} in the mid 1980s. In 1988--92, Bousch \cite{Bo1,Bo2} proved $\mathcal{M}$ and $\mathcal{M}_0$ are connected and locally connected. In 2002, Bandt \cite{Ba} proved the existence of a hole in $\mathcal{M}$, and conjectured that there are in fact infinitely many holes. A different point of view was presented in 1993, where Odlyzko and Poonen \cite{OdPo} proved connectedness of the closure of, and conjectured the presence of holes in, the related set of roots of polynomials with $0,1$ coefficients. In 2014 Calegari, Koch, and Walker \cite{CKW} gave a positive answer to the conjecture, for both $\mathcal{M}$ and $\mathcal{M}_0$, and in fact constructed infinite sequences of holes accumulating at certain parameters in $\partial\mathcal{M}$ (see Figure~\ref{fig:buried}). These are buried points of $\partial\mathcal{M}$, i.e. not path accessible from the complement $\mathbb{D}\setminus\mathcal{M}$.

The question of classifying these connected components of $\mathbb{D}\setminus\mathcal{M}$ is still open; however our results are a step in that direction. Let $\mathcal{P}$ denote the set of all normalized power series with coefficients in $\{-1,0,+1\}$, i.e. \[\mathcal{P}:=\left\{f(z)=\sum_{j=0}^\infty c_jz^j~\bigg|~c_j\in\{-1,0,+1\},~c_0=1\right\}.\]
With this notation our main results are as follows:
\begin{theorem}
Suppose $f$ is the unique power series in $\mathcal{P}$ that vanishes at $\lambda\in\mathcal{M}\setminus\mathbb{R}$ with $\left|\lambda\right|\leq2^{-1/2}$. If $f$ has finitely many zero coefficients and its Taylor polynomials satisfy certain conditions then $\lambda$ is on the boundary of a hole of $\mathcal{M}$.
\end{theorem}
Furthermore, with minor adjustments to the conditions on the Taylor polynomials,
\begin{theorem}
Suppose $f$ is the unique power series in $\mathcal{P}$ that vanishes at $\lambda\in\mathcal{M}\setminus\mathbb{R}$ with $\left|\lambda\right|\leq2^{-1/2}$. If $f$ has no zero coefficients and its Taylor polynomials satisfy certain conditions then $\lambda$ is on the boundary of a hole of $\mathcal{M}_0$.
\end{theorem}
For the complete formulation, see Theorems~\ref{thm:bdholem} and \ref{thm:bdholem0}.

Further motivation to study points of $\partial\mathcal{M}_0\cap\partial\mathcal{M}$ is provided by the results in \cite{ERS,HaisP}. It is possible to derive a quadratic-like map $g$ from the maps $\mathfrak{s}_-$ and $\mathfrak{s}_+$, so that the dynamics of $g$ on the limit set $\mathsf{A}_\lambda$ is quasisymmetrically conjugate to the one of $z^2+c$ on the Julia set for some Misiurewicz parameter $c$. In fact, \begin{conjecture}
There is a set of parameters $\lambda$, dense in $(\partial\mathcal{M}\cap\partial\mathcal{M}_0)\setminus\mathbb{R}$, so that the dynamics of $g$ on the limit set $\mathsf{A}_\lambda$ is quasisymmetrically conjugate to that of $z^2+c$ on its Julia set for some Misiurewicz parameter $c$.
\end{conjecture}

This article is organized as follows. In Section~\ref{sect:notation} we set the notation, while in Section~\ref{sect:ovrlp} we describe the sets $\mathcal{M}$ and $\mathcal{M}_0$ as the closure of the roots of power series with certain prescribed coefficients. We overview some of the known results about $\mathcal{M}$ and $\mathcal{M}_0$ in Section~\ref{sect:ImpThms} and explain how our results fit in. The proofs are provided in Section~\ref{sect:pfMainThm}. Finally, Section~\ref{sect:ThmExamples} is devoted to salient examples of parameters $\lambda$ to which our theorems apply.

\section{Notation}
\label{sect:notation}

Let $\Sigma^n$ be the set of all words $\mathtt{\bm w}=a_0\cdots a_{n-1}$ of length $n$ from the alphabet $\{-,+\}$. Define $\mathfrak{s}_{\mathtt{\bm w}}=\mathfrak{s}_{a_0}\circ\ldots\circ\mathfrak{s}_{a_{n-1}}$ so that  $\mathsf{A}_\lambda^\mathtt{\bm w}=\mathfrak{s}_{\mathtt{\bm w}}(\mathsf{A}_\lambda)$ and
\[\mathsf{A}_\lambda=\bigcup_{\mathtt{\bm w}\in\Sigma^n}\mathsf{A}_\lambda^\mathtt{\bm w}.\]
The natural projection from the space $\Sigma^\infty$ of infinite words, $\mathtt{\bm w}=a_0a_1\cdots$ in $\{-,+\}$, onto the attractor, $\mathsf{A}_\lambda$ is given by 
\[\pi_\lambda:\Sigma^\infty\to\mathsf{A}_\lambda\,,\qquad \pi_\lambda(\mathtt{\bm w})=\sum_{n=0}^\infty a_n\lambda^n.\]
The word $\mathtt{\bm w}\in\Sigma^\infty$ is called the {\em itinerary} of $\pi_\lambda(\mathtt{\bm w})$ in $\mathsf{A}_\lambda$. Remember that if $\mathsf{A}_\lambda$ is connected, then there could be multiple itineraries for a given point in $\mathsf{A}_\lambda$.

Nevertheless, elements $\omega\in\mathsf{A}_\lambda$ that are in $\mathsf{A}_\lambda^{a_0a_1\cdots a_k}$ can be written as
\[\omega=a_0+a_1\lambda+\ldots+a_k\lambda^k+\sum_{n=k+1}^\infty a_n\lambda^n,\qquad a_n\in\{-1,+1\}.\]
We will use $\mathtt{\bm w}|k$ to denote the finite word $a_0a_1\cdots a_k$ coming from the truncation of the infinite word $\mathtt{\bm w}=a_0a_1\cdots$. The notation $\left|\mathtt{\bm w}\right|$ will indicate the length of the word $\mathtt{\bm w}$.

Let $D_r(z)$ denote the closed disk centered at $z$ with radius $r$ and, when $z=0$, we will abbreviate $D_r=D_r(0)$. The following lemma is straightforward.
\begin{lemma}
\label{lem:diambdd}
   Let $R\geq(1-\left|\lambda\right|)^{-1}$, then $\mathsf{A}_\lambda\subset D_R$ and $\widetilde{\mathsf{A}}_\lambda\subset D_R$.
\end{lemma}
From here on, we let $R:=(1-\left|\lambda\right|)^{-1}$. Let $\mathsf{I}^{-1}=\widetilde{\mathsf{I}}^{-1}$ be the disk $D_R$, and consider the recursive constructions $\mathsf{I}^n=\mathfrak{s}_-(\mathsf{I}^{n-1})\cup\mathfrak{s}_+(\mathsf{I}^{n-1})$, and  $\widetilde{\mathsf{I}}^n=\mathfrak{s}_-(\widetilde{\mathsf{I}}^{n-1})\cup\mathfrak{s}_{\mathtt{O}}(\widetilde{\mathsf{I}}^{n-1})\cup\mathfrak{s}_+(\widetilde{\mathsf{I}}^{n-1})$ for $n\in\mathbb{N}$. From Lemma~\ref{lem:diambdd}, it is clear that 
\[ \mathsf{A}_\lambda=\bigcap_{n=0}^\infty \mathsf{I}^n \quad\text{ and }\quad \widetilde{\mathsf{A}}_\lambda=\bigcap_{n=0}^\infty \widetilde{\mathsf{I}}^n.\]
As before, for any finite word $\mathtt{\bm w}=a_0a_1\cdots a_k$ in $\Sigma^{k+1}$ we can identify the disks in $\mathsf{I}^{k+1}$ as $\mathsf{D}^{\mathtt{\bm w}}:=\mathfrak{s}_{\mathtt{\bm w}}(\mathsf{I}^{-1})$. Each of these disks will then be centered at $\mathfrak{s}_{\mathtt{\bm w}}(0)=a_0+a_1\lambda+\cdots+a_k\lambda^k$, and have a radius of $\left|\lambda\right|^{k+1}R$. If $\mathtt{\bm w}\in\Sigma^\infty$ is an infinite word, then $\mathfrak{s}_{\mathtt{\bm w}}(0)\in\mathsf{A}_\lambda^{\mathtt{\bm w}|k}\subset \mathsf{D}^{\mathtt{\bm w}|k}$ for all $k\geq0$. Analogously, if $\mathtt{\bm w}\in\widetilde{\Sigma}^\infty$, then $\mathfrak{s}_{\mathtt{\bm w}}(0)\in\widetilde{\mathsf{A}}_\lambda^{\mathtt{\bm w}|k}\subset \widetilde{\mathsf{D}}^{\mathtt{\bm w}|k}$ for all $k\geq0$.

From now on, we will refer to $\mathsf{I}^k$ and $\widetilde{\mathsf{I}}^k$ as the {\em instar}\footnote{
The word ``instar'' is used in biology to describe the developmental stage of insects, between each molt until sexual maturity. We chose it because the limit set is obtained by going through (infinitely many) developmental stages, some with a definite larva-like appeareance.} at level $k$ of the IFS $\{\mathfrak{s}_-,\mathfrak{s}_+\}$ and $\{\mathfrak{s}_-,\mathfrak{s}_{\mathtt{O}},\mathfrak{s}_+\}$, respectively. Given that $\mathsf{I}^k$ is the union of two copies of the instar at level $k-1$, $\mathfrak{s}_+(\mathsf{I}^{k-1})$ and $\mathfrak{s}_-(\mathsf{I}^{k-1})$ will be respectively called the {\em positive} and {\em negative instars} at level $k$. Similarly, we refer to $\mathfrak{s}_+(\widetilde{\mathsf{I}}^{k-1})$, $\mathfrak{s}_{\mathtt{O}}(\widetilde{\mathsf{I}}^{k-1})$, and $\mathfrak{s}_-(\widetilde{\mathsf{I}}^{k-1})$ as the {\em positive}, {\em central}, and {\em negative instar} of level $k$.

It will also be useful to have a name for each of the disks and their centers in the instar. The center $\mathfrak{s}_{\mathtt{\bm w}|k}(0)=\sum_{j=0}^k a_j\lambda^j$ will be called a {\em node} with itinerary $\mathtt{\bm w}|k$ and generally denoted by $\nu_{\mathtt{\bm w}|k}$, while the closed disk $\mathsf{D}^{\mathtt{\bm w}|k}$ (or $\widetilde{\mathsf{D}}^{\mathtt{\bm w}|k}$) centered there will be referred to as a {\em nodal disk} with itinerary $\mathtt{\bm w}|k$.

\section{The Overlap Set}
\label{sect:ovrlp}
The geometric structure of the attractor $\mathsf{A}_\lambda$ is determined in many ways by the overlap set $O_\lambda:=\mathsf{A}_\lambda^-\cap \mathsf{A}_\lambda^+$. It is a standard exercise to show that if $O_\lambda$ is empty then $\mathsf{A}_\lambda$ is simply a Cantor set and consequently $\lambda\in\mathbb{D}\setminus\mathcal{M}$. However, if $O_\lambda$ is ``large'' then it becomes difficult to distinguish the smaller affine copies that constitute $\mathsf{A}_\lambda$. Moreover, if the large size of the overlap is persistent through a small change of $\lambda$, then the parameter is in the interior of $\mathcal{M}$.
Intuitively, $\lambda\in\partial\mathcal{M}$ whenever $O_\lambda$ is in some sense ``thin''.

Whenever $O_\lambda$ is nonempty there exist itineraries $\mathtt{\bm a},\mathtt{\bm b}\in\Sigma^\infty$ with $a_0=-$ and $b_0=+$ such that
\[
\pi_\lambda(\mathtt{\bm w}):=\sum_{j=0}^\infty a_j\lambda^j=\sum_{j=0}^\infty b_j\lambda^j=:\pi_\lambda(\mathtt{\bm v})\iff \sum_{j=0}^\infty (a_j-b_j)\lambda^j=0
\]
Observe that $a_j-b_j\in\{-2,0,+2\}$ for every $j\geq0$. Consequently,  we can ignore a factor of $2$, consider the set
\[\mathcal{P}=\left\{f(z)=\sum_{j=0}^\infty c_jz^j~\bigg|~c_j\in\{-1,0,+1\},~c_0=1\right\},\]
and define the set of power series which have $\lambda$ as a root by
\[\mathcal{F}_\lambda := \{f\in\mathcal{P}~|~ f(\lambda)=0\}.\]

\begin{note}
With some variations, the notation for attractors, parameter spaces, and spaces of power series functions has been used for some time (\hspace{1sp}\cite{BarHar}, \cite{Bo1}, \cite{Ba}, \cite{Sol}, \cite{ERS}, and more recently,~\cite{CKW}). The reader should be aware that the spaces $\mathcal{P}$ and $\mathcal{F}_{\lambda}$ are used both in the context of $2$- and of $3$-attractors, so adding a tilde to signal the use of three values $\{-1,0,1\}$ would only make the notation more cumbersome.
\end{note}

\begin{note}\label{note:2}
In particular, if there is an overlap point $\xi := \pi_\lambda(\mathtt{\bm w})=\pi_\lambda(\mathtt{\bm v})$, and $f(z)=\sum_{j=0}^\infty c_j z^j$ with $c_j=(a_j-b_j)/2$, then $\xi$ satisfies
\[
  \sum_{j=0}^\infty a_j\lambda^j =
  \sum_{j=0}^\infty c_j\lambda^j+\sum_{\substack{0<j\leq \infty\\ \text{s.t. }c_j=0}} a_j\lambda^j =
  f(\lambda)+\sum_{\substack{0<j\leq \infty\\ \text{s.t. }c_j=0}} a_j\lambda^j =
  \sum_{\substack{0<j\leq \infty\\ \text{s.t. }c_j=0}} a_j\lambda^j.
\]
The overlap set, $O_\lambda$ is nonempty whenever the function $f\in\mathcal{P}$ with coefficients $c_j=(a_j-b_j)/2$ satisfies $f(\lambda)=0$. Conversely, if for a particular $\lambda\in\mathbb{D}$ the set $\mathcal{F}_\lambda$ is nonempty, then so is $O_\lambda$, and each element in it has an itinerary associated to some $f\in\mathcal{F}_\lambda$. We have just shown that 
\[\mathcal{M}=\{\lambda\in\mathbb{D}~|~\mathsf{A}_\lambda\text{ is connected}\} = \{\lambda\in\mathbb{D}~|~\left|\mathcal{F}_\lambda\right|\neq0\}.\]
\end{note}

Bousch \cite{Bo1,Bo2} observed that since $\mathsf{A}_\lambda$ is symmetric with respect to $0$, having the origin in the overlap implies that the coefficients $c_j$ of at least one of the power series $f\in\mathcal{F}_\lambda$ must all be nonzero (see Lemma~\ref{lem:onePOM}). It follows that 
\[\mathcal{M}_0=\{\lambda\in\mathbb{D}~|~0\in\mathsf{A}_\lambda \}=\left\{\lambda\in\mathbb{D}~\bigg|~\exists f\in\mathcal{P},~f(\lambda)=\sum_{j=0}^\infty c_j\lambda^j=0,~~c_j\in\{-1,1\}\right\}\]
from which it is clear that $\mathcal{M}_0\subset\mathcal{M}$.

The following result is important, as it gives more insight on the relationship between elements in $O_\lambda$ and the power series which have $\lambda$ as a root.

\begin{lemma}[Solomyak \cite{Sol}]
\label{lem:onePOM}
$\left|O_\lambda\right|=1$ or $2$ if and only if $\left|\mathcal{F}_\lambda\right|=1$. Moreover,
\begin{enumerate}[{\bf (i)}]
    \item $\left|O_\lambda\right|=1$ if and only if $f\in \mathcal{F}_\lambda$ has no zero coefficients.
    \item $\left|O_\lambda\right|=2$ if and only if $f\in \mathcal{F}_\lambda$ has exactly one zero coefficient.
\end{enumerate}
\end{lemma}

Using Rouche's Theorem and careful estimates Solomyak was also able to prove that
\begin{theorem}[Solomyak \cite{Sol}]
\label{thm:uncountLa}
There exist uncountably many $\lambda\in\mathcal{M}$ for which $\left|O_\lambda\right|=1$. The itinerary of $0\in\mathsf{A}_\lambda$ is different for different $\lambda$.
\end{theorem}

There is a historically important property of an IFS which ensures that there is not ``too much'' overlap.
We say that an IFS of contraction similarities $\left\{\mathfrak{s}_{j}\right\}_{j=1}^{m}$ satisfies the {\em open set condition} (OSC) if there exists a nonempty open set $V\subset\mathbb{C}$ with \[\bigcup_{j=1}^m\mathfrak{s}_{j}(V)\subseteq V\text{ and } \mathfrak{s}_{j}(V)\cap \mathfrak{s}_{k}(V)=\emptyset \text{ for all } j\neq k.\]

The example of $\lambda={\rm i}/\sqrt{2}$ is useful in gaining intuition about the OSC. The attractor for $\mathsf{A}_{\lambda}$ is the rectangle with corners $\pm2\pm\sqrt{2}{\rm i}$, with side ratio $\sqrt{2}$. The two images $\mathfrak{s}_{+}(\mathsf{A}_\lambda)$, $\mathfrak{s}_{-}(\mathsf{A}_\lambda)$ cover the left and right halves of $\mathsf{A}_\lambda$, much as a A$4$ sheet of paper folded in half. Their intersection is the middle fold, so a feasible open set is the interior of $\mathsf{A}_\lambda$ itself.

In general, it can be challenging to prove that the OSC holds for a given IFS with a connected attractor. Bandt and his collaborators have recently shown that for connected self-similar sets in the plane a finite overlap implies OSC \cite{BaRa}. Specifically, they showed the following:
\begin{theorem}[Bandt-Hung \cite{BaHu}]
\label{thm:uncountLaZ}
For every $m\in\mathbb{N}$ there are uncountably many $\lambda\in\mathcal{M}$ for which OSC holds, and the overlap set consists of $2^m$ points. For each such $\lambda$ there exists a unique and distinct $f\in\mathcal{P}$ such that $\mathcal{F}_\lambda=\{f\}$.
\end{theorem}
\begin{theorem}[Bandt-Hung \cite{BaHu}]
\label{thm:uncountLaCant}
For every $\beta\in[0,0.2]$ there are uncountably many $\lambda\in\mathcal{M}$ for which OSC holds, and the overlap set is a Cantor set of Hausdorff dimension $\beta$. For each $\lambda$ there exists a unique and distinct $f\in\mathcal{P}$ such that $\mathcal{F}_\lambda=\{f\}$.
\end{theorem}
It must be noted that the proof of Lemma~\ref{lem:onePOM} cannot be easily extended to the case of $\left|O_\lambda\right|=2^m$ for $m\geq2$. Indeed, Bandt and Hung used a different argument to show the uniqueness of the power series.

\section{Self and Asymptotic Similarity}
\label{sect:ImpThms}
Before describing the old and new results about $\partial\mathcal{M}$, we recall some definitions which can be found in \cite{L}. Remember that $D_r(z)$ denotes a closed disk centered at $z$ with radius $r$ and $D_r=D_r(0)$. For compact sets $E,F\subset\mathbb{C}$ denote \[[E]_r=(E\cap D_r)\cup \partial D_r; \qquad {\rm d}_r(E,F)={\rm d_H}([E]_r,[F]_r)\] where ${\rm d_H}$ is the Hausdorff distance.
\begin{definition}{\ }
    \begin{enumerate}[{\bf(i)}]
        \item A compact set $F$ is {\em $\rho$-self-similar} about $z\in F$, for $\rho\in\mathbb{C}\setminus\overline{\mathbb{D}}$, if there is $r>0$ such that $[\rho(F-z)]_r=[F-z]_r$.
        \item Two compact sets $E$ and $F$ are {\em asymptotically similar} about $z\in E$ and $w\in F$ if there is $r>0$ such that
        \[\lim_{t\in\mathbb{C},~\left|t\right|\to\infty} {\rm d}_r(t(E-z),t(F-w))=0.\]
        \item A compact set $E$ is {\em asymptotically $\rho$-self similar} about a point $z\in E$ if there is $r>0$ and a compact set $F$ such that 
        \[{\rm d}_r(\rho^n(E-z),F)\to0\quad\text{ as }n\to\infty.\]
    \end{enumerate}
\end{definition}

\begin{definition}
A power series $f\in\mathcal{P}$ is said to be {\em of rational type} $(\ell,p)$ if $\ell$ and $p$ are the minimal integers such that
\begin{equation}
 \label{f-rat}
  f(z)=
  \sum_{j=0}^\ell c_jz^j+\frac{c_{\ell+1}z^{\ell+1}+\ldots+c_{\ell+p}z^{\ell+p}}{1-z^p},\qquad \big( c_j\in\{-1,0,+1\} \big).\tag{$\star$}
\end{equation}
(Note that this is equivalent to the requirement that the sequence of coefficients is pre-periodic)
\end{definition}

We can now state the result of Solomyak:
\begin{theorem}[Solomyak \cite{Sol}]
\label{thm:asymSol}
    Suppose $\lambda\in\mathcal{M}\setminus\mathbb{R}$, with $\left|\lambda\right|\leq2^{-1/2}$, is such that $\mathcal{F}_\lambda=\{f\}$ with $f$ of rational type $(\ell,p)$. Then $f'(\lambda)\neq0$ and
    \begin{enumerate}[{\bf(i)}]
        \item $\widetilde{\mathsf{A}}_{\lambda}$ is $\lambda^{-p}$-self similar about  $-\lambda^{-(\ell+1)}\sum_{n=0}^\ell c_j\lambda^j=:\zeta$.
        \item $\mathcal{M}$ about $\lambda$ is asymptotically similar to $\tfrac{\lambda^{\ell+1}}{f'(\lambda)}\widetilde{\mathsf{A}}_{\lambda}$ about $\tfrac{\lambda^{\ell+1}}{f'(\lambda)}\zeta$.
        \item $\mathcal{M}$ is asymptotically $\lambda^{-p}$-self similar about $\lambda$.
    \end{enumerate}
\end{theorem}

Notice that if the coefficients of $f$ are all non zero, then the theorem holds true if we substitute $\widetilde{\mathsf{A}}_\lambda$ with $\mathsf{A}_\lambda$ and $\mathcal{M}$ with $\mathcal{M}_0$. However, Theorem~\ref{thm:asymSol} is not enough to certify that parameters $\lambda$ satisfying the hypothesis lie on $\partial\mathcal{M}$, because points in a neighborhood of $\zeta$, not in $\widetilde{\mathsf{A}}_\lambda$, are not necessarily also outside of $\mathcal{M}$ in a neighborhood of $\lambda$.

\begin{theorem}[Calegari-Koch-Walker \cite{CKW}]
\label{thm:asymCKW}
    Suppose $\lambda\in\mathcal{M}\setminus\mathbb{R}$, with $\left|\lambda\right|\leq2^{-1/2}$, is a root of   $f$ of rational type $(\ell,p)$.
    \begin{enumerate}[{\bf(i)}]
    \item If $C\in \tfrac{\lambda^{\ell+1}}{f'(\lambda)}\left(\widetilde{\mathsf{A}}_{\lambda}-\zeta\right)$, then for every $\varepsilon>0$, there is a $C'$ such that $\left|C-C'\right|<\varepsilon$ and for all sufficiently large $n$, a neighborhood of $C'\lambda^{pn}+\lambda$ is contained in $\mathcal{M}$.
    \item If $\mathcal{F}_\lambda=\{f\}$, then there is $\delta>0$ such that for every $C\not\in \tfrac{\lambda^{\ell+1}}{f'(\lambda)}\left(\widetilde{\mathsf{A}}_{\lambda}-\zeta\right)$ with $|C|<\delta$, the parameter $C\lambda^{pn}+\lambda$ is not in $\mathcal{M}$ for all sufficiently large $n$.
    \end{enumerate}
\end{theorem}

Observe that this theorem is more descriptive than Theorem~\ref{thm:asymSol}{\bf (ii)} as it describes explicitly which neighborhoods of $\zeta\in\widetilde{\mathsf{A}}_\lambda$ converge in the Hausdorff metric to neighborhoods of $\lambda\in\mathcal{M}$ (see Figure~\ref{fig:buried}). However, this result gives little information about the local topology of $\mathcal{M}$ around $\lambda$. In particular, the question of recognizing points of $\partial\mathcal{M}$ path accessible from $\mathbb{D}\setminus \mathcal{M}$ remains open. Our main result gives a partial answer. 
\begin{remark}
In the following statements it is assumed that $\lambda$ is not real with $\left|\lambda\right|\leq2^{-1/2}$. We will call {\em accessible} those points of $\partial\mathcal{M}$ (respectively $\partial\mathcal{M}_0$) path accessible from $\mathbb{D}\setminus \mathcal{M}$ (respectively $\mathbb{D}\setminus\partial\mathcal{M}_0$).
\end{remark}

\begin{figure}
    \centering
    \begin{picture}(205,205)
    \put(0,0){\includegraphics[scale=0.3]{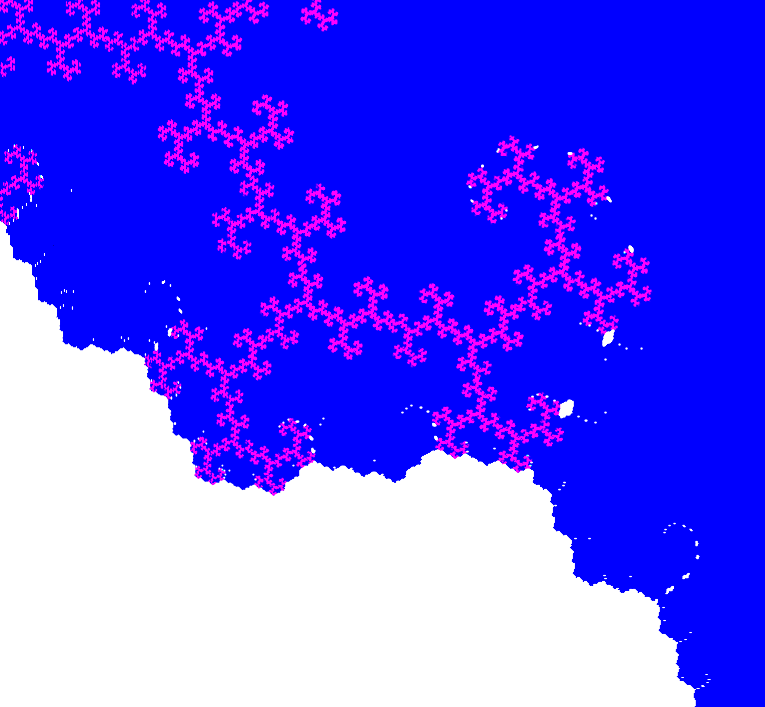}}
    \put(150,148){{\color{cyan}\textbullet}}
    \end{picture}
    \caption{Calegari et al. showed that $\lambda\approx0.371859 + 0.519411{\rm i}$ (marked in cyan) is a point of $\partial\mathcal{M}\cap\partial\mathcal{M}_0$ not path accessible from $\mathbb{D}\setminus\mathcal{M}$.}
    \label{fig:buried}
\end{figure}
\begin{theorem}
\label{thm:bdholem}
    Let $\lambda\in\mathcal{M}$ be such that $\mathcal{F}_\lambda=\{f\}$ with $f$ of rational type $(\ell,p)$.
    Assume also that $f(z)=\sum_{j=0}^\infty c_jz^j$ has finitely many zero coefficients and that its Taylor polynomials, $f_k(z)=\sum_{j=0}^kc_jz^j$ satisfy the following conditions for every $0\leq n\leq p-1$:
    \begin{enumerate}[{\bf(i)}]
        \item $\displaystyle \left|f_{\ell+1+n}(\lambda)\right|>\frac{1}{2}\frac{\left|\lambda^{\ell+1+n+1}\right|}{1-\left|\lambda\right|};$
        \item $\displaystyle \left|f_{\ell+1+n}(\lambda)\right|+\left|f_{\ell+1+n+1}(\lambda)\right|>\frac{\left|\lambda^{\ell+1+n+1}\right| }{1-\left|\lambda\right|};$
        \item Let $Q$ be the solution of $2f_\ell(z)+z^{\ell+1}Q(z)=2f_{\ell+1+n}(z)$. Then for any polynomial $P$ of degree less or equal than $n$ with coefficients in $\{-2,-1,0,+1,+2\}$, with $P\neq Q$ it holds that $\displaystyle 2\left|f_{\ell+1+n}(\lambda)\right|<\left|2f_\ell(\lambda)+\lambda^{\ell+1}P(\lambda)\right|$.
    \end{enumerate}
    Then $\lambda\in\partial\mathcal{M}$ is accessible from a connected component of $\mathbb{D}\setminus\mathcal{M}$.
\end{theorem}
An immediate corollary is then 
\begin{corollary}
\label{cor:sharedBd}
  Let $\lambda\in\mathcal{M}$ be such that $\mathcal{F}_\lambda=\{f\}$ where $f$ satisfies the hypothesis of Theorem~\ref{thm:bdholem} and has no zero coefficients. Then $\lambda\in\partial\mathcal{M}\cap\partial\mathcal{M}_0$ and it is accessible from a connected component of $\mathbb{D}\setminus\mathcal{M}$.
\end{corollary}
Restricting the assumptions on the coefficients on the power series $f$, we also obtain
\begin{theorem}
\label{thm:bdholem0}
    Let $\lambda\in\mathcal{M}$ be such that $\mathcal{F}_\lambda=\{f\}$ with $f$ of rational type $(\ell,p)$.
    Assume also that $f(z)$ has no zero coefficients and that its Taylor polynomials, $f_k(z)=\sum_{j=0}^kc_jz^j$ satisfy conditions {\bf(i)} and {\bf(ii)} of Theorem~\ref{thm:bdholem}, together with
    \begin{itemize}
        \item[{\bf(iii')}] Let $Q$ be the solution of $f_\ell(z)+z^{\ell+1}Q(z)=f_{\ell+1+n}(z)$. Then for any polynomial $P$ of degree less or equal than $n$ with coefficients in $\{-1,0,+1\}$, with $P\neq Q$ it holds that $\displaystyle \left|f_{\ell+1+n}(\lambda)\right|<\left|f_\ell(\lambda)+\lambda^{\ell+1}P(\lambda)\right|$.
    \end{itemize}
    Then $\lambda\in\partial\mathcal{M}_0$ is accessible from a connected component of $\mathbb{D}\setminus\mathcal{M}_0$.
\end{theorem}

In view of Corollary~\ref{cor:sharedBd} and experimental evidence, it seems reasonable that the assumptions of Theorem~\ref{thm:bdholem0} imply that condition {\bf(iii)} of Theorem~\ref{thm:bdholem} must hold. Even better:
\begin{conjecture}
Every parameter $\lambda\in\partial\mathcal{M}\cap\partial\mathcal{M}_0$ belongs to the boundary of the connected component of $\mathbb{D}\setminus\mathcal{M}_0$ containing $0$. 
\end{conjecture}

\section{Proof of the Main Theorems}
\label{sect:pfMainThm}
Here we prove Theorems \ref{thm:bdholem} and \ref{thm:bdholem0} in this section. In both cases, the idea of the proof is to construct locally, a connected chain of open disks outside $\widetilde{\mathsf{A}}_\lambda$ (or $\mathsf{A}_\lambda$) that converges to $\zeta=-\lambda^{-(\ell+1)}\sum_{j=0}^\ell c_j\lambda^j$, and conclude by Theorem \ref{thm:asymCKW}{\bf (ii)} that $\lambda$ is accessible, hence on the boundary of a hole. The restrictions on the Taylor polynomials show up when obtaining the conditions for such a chain to exists.

We first need a lemma
\begin{lemma}
\label{lem:selfsimD}
Let $\lambda\in\mathcal{M}$ be such that $\mathcal{F}_\lambda=\{f\}$ with $f$ of rational type $(\ell,p)$, and where $c_j$ can hold the value $0$ only in the range $0<j\leq \ell$. Let $\xi\in O_\lambda$ have itineraries $\mathtt{\bm a}=a_0a_1a_2\cdots$ and $\overline{\mathtt{\bm a}}=\overline{a}_0\overline{a}_1\overline{a}_2\cdots$ where $a_j=-\overline{a}_j=c_j$ if $c_j\neq0$, and otherwise $a_j=\overline{a}_j=-$ or $+$.\\
Then for every $0\leq n\leq p-1$ the sets $\big( \mathsf{D}^{\mathtt{\bm a}|\ell+n+p}\cup \mathsf{D}^{\overline{\mathtt{\bm a}}|\ell+n+p} \big)$ and $\big( \mathsf{D}^{\mathtt{\bm a}|\ell+n}\cup \mathsf{D}^{\overline{\mathtt{\bm a}}|\ell+n} \big)$ are similar about the common point $\xi$, with scaling factor $\lambda^{-p}$.
\end{lemma}

\begin{proof}
    Since the set in question is the union of two intersecting closed disks, proving the lemma is equivalent to showing that
    \[\frac{1}{\lambda^p}\left(\left(\mathsf{D}^{\mathtt{\bm a}|\ell+n+p}\cup \mathsf{D}^{\overline{\mathtt{\bm a}}|\ell+n+p}\right)-\xi\right)=\left(\mathsf{D}^{\mathtt{\bm a}|\ell+n}\cup \mathsf{D}^{\overline{\mathtt{\bm a}}|\ell+n}\right)-\xi.\]
    In order to simplify the expressions, we only prove this in the case $n=0$; but the argument is identical for $0<n\leq p-1$.

    Observe that since $\lambda$ is a root of $f(z)$, formula \eqref{f-rat} can be re-written
    \[ c_{\ell+1}\lambda^{\ell+1}+\ldots+c_{\ell+p}\lambda^{\ell+p}=(\lambda^p-1)(c_0+c_1\lambda+\ldots+c_\ell\lambda^\ell)\]
    or equivalently
\[c_0+c_1\lambda+\ldots+c_{\ell+p}\lambda^{\ell+p}=\lambda^p(c_0+c_1\lambda+\ldots+c_\ell\lambda^\ell).\]
    Consequently,
    {\small \begin{align*}
    \left|(a_0+a_1\lambda+\ldots+a_{\ell+p}\lambda^{\ell+p})-(\overline{a}_0+\overline{a}_1\lambda+\ldots+\overline{a}_{\ell+p}\lambda^{\ell+p})\right|&=2\left|c_0+c_1\lambda+\ldots+c_{\ell+p}\lambda^{\ell+p}\right|\\
    &=2\left|\lambda^p\right|\left|c_0+c_1\lambda+\ldots+c_\ell\lambda^\ell\right|
    \end{align*}
    }
    and 
    {\small \begin{eqnarray*}
    \left|(a_0+a_1\lambda+\ldots+a_\ell\lambda^\ell)-(\overline{a}_0+\overline{a}_1\lambda+\ldots+\overline{a}_\ell\lambda^\ell)\right|&=&2\left|c_0+c_1\lambda+\ldots+c_\ell\lambda^\ell\right|
    \end{eqnarray*}
    }
    which shows that the distance of the nodes at levels $\ell+p$ and $\ell$, are multiples of each other.
    
    Finally, recall that we assume $a_j=c_j$ whenever $c_j\neq0$ so the center of the disk $\mathsf{D}^{a_0a_1\cdots a_{\ell+p}}$ can be written as
    \begin{eqnarray*}
    \sum_{j=0}^{\ell+p}a_{j}\lambda^{j}&=&\sum_{\substack{0<j\leq \ell\\ \text{s.t. }c_j=0}}a_j\lambda^j+c_0+c_1\lambda+\ldots+c_{\ell+p}\lambda^{\ell+p}\\
    &=&\sum_{\substack{0<j\leq \ell\\ \text{s.t. }c_j=0}}a_j\lambda^j+\lambda^p(c_0+c_1\lambda+\ldots+c_\ell\lambda^\ell).
    \end{eqnarray*}
    In other words,
    \begin{eqnarray*}
    \sum_{j=0}^{\ell+p}a_j\lambda^j-\sum_{\substack{0<j\leq \ell\\ \text{s.t. }c_j=0}}a_j\lambda^j&=&\lambda^p\sum_{j=0}^\ell c_j\lambda^j=\lambda^p\left(\sum_{j=0}^\ell a_j\lambda^j-\sum_{\substack{0<j\leq \ell\\ \text{s.t. }c_j=0}}a_j\lambda^j\right).
    \end{eqnarray*}
    Then, since $\displaystyle \xi=\sum_{\substack{0<j\leq \ell\\ \text{s.t. }c_j=0}}a_j\lambda^j$ (cf. Note in Section~\ref{sect:ovrlp}), the above equation becomes
    \[a_0+a_1\lambda+\ldots+a_{\ell+p}\lambda^{\ell+p}-\xi=\lambda^p(a_0+a_1\lambda+\ldots+a_{\ell}\lambda^{\ell}-\xi)\]
    that is 
    \[\frac{1}{\lambda^p}\left(\mathsf{D}^{\mathtt{\bm a}|\ell+p}-\xi\right)=\left(\mathsf{D}^{\mathtt{\bm a}|\ell}-\xi\right).\]
    Now, the disk $\mathsf{D}^{\mathtt{\overline{\bm a}}|\ell+k}$ is symmetric to $\mathsf{D}^{\mathtt{\bm a}|\ell+k}$ relative to $\xi$ for any $k\geq0$. Hence, analogous arguments holds for the disks $\mathsf{D}^{\mathtt{\overline{\bm a}}|\ell+p}$ and $\mathsf{D}^{\mathtt{\overline{\bm a}}|\ell}$.
\end{proof}

Recall from Lemma~\ref{lem:diambdd} that $\mathsf{A}_\lambda\subset D_R$ and $\mathsf{A}_\lambda^\mathtt{\bm w}\subset\mathsf{D}^\mathtt{\bm w}$ for any finite word $\mathtt{\bm w}\in \Sigma^n$. Consequently, the above lemma proves the self-similarity of the attractor $\mathsf{A}_\lambda$ at its overlap. Concretely, we find two subsets of $\mathsf{A}_{\lambda}$, centered around the common point $\xi$, which scale into each other:
\begin{corollary}
\label{cor:selfsimA}
Let $\lambda\in\mathcal{M}$ be such that $\mathcal{F}_\lambda=\{f\}$ with $f$ of rational type $(\ell,p)$, and $c_j=0$ only for some $0<j\leq \ell$. Let $\xi\in O_\lambda$ have itineraries $\mathtt{\bm a}=a_0a_1a_2\cdots$ and $\overline{\mathtt{\bm a}}=\overline{a}_0\overline{a}_1\overline{a}_2\cdots$ where $a_j=-\overline{a}_j=c_j$ if $c_j\neq0$, and otherwise $a_j=\overline{a}_j=-$ or $+$.\\
Then for every $0\leq n\leq p-1$ and every integer $k\geq1$ the sets $\big( \mathsf{A}_\lambda^{\mathtt{\bm a}|\ell+n+kp}\cup \mathsf{A}_\lambda^{\overline{\mathtt{\bm a}}|\ell+n+kp} \big)$ and $\big( \mathsf{A}_\lambda^{\mathtt{\bm a}|\ell+n}\cup \mathsf{A}_\lambda^{\overline{\mathtt{\bm a}}|\ell+n} \big)$ are similar about $\xi$, with scaling factor $\lambda^{-kp}$.
\end{corollary}

We now proceed to the construction of the chain in the complement of $\widetilde{\mathsf{A}}_\lambda$. We exploit the recursive construction of $\widetilde{\mathsf{A}}_\lambda$ to find each disk in the chain: for each $n\geq0$ we find an open disk tangent to the instar $\widetilde{\mathsf{I}}^n$. Moreover, two consecutive disks in the chain must intersect non trivially. Finally, this chain must converge to $\zeta\in\widetilde{\mathsf{A}}_\lambda$.

The parameter $\lambda$ is the root of a unique power series $f\in\mathcal{P}$ whose non-zero coefficients eventually repeat. That is, $f$ is of rational type $(\ell,p)$:
\[f(z)=\sum_{j=0}^\ell c_jz^j+\frac{c_{\ell+1}z^{\ell+1}+\ldots+c_{\ell+p}z^{\ell+p}}{1-z^p}.\]
Now, $\zeta$ is defined to be $-\lambda^{-(\ell+1)}\sum_{j=0}^\ell c_j\lambda^j$, which means (by the above equation) that it can be described with the periodic itinerary $(c_{\ell+1}\cdots c_{\ell+p})^\infty\in\Sigma^\infty$. This itinerary will be the unique one associated to $\zeta$ as long as $\left|\mathcal{F}_\lambda\right|=1$. Since we assume so in the statement of the theorems, there is a unique sequence of nodes converging to $\zeta$,  namely the ones whose itinerary is the truncation of $(c_{\ell+1}\cdots c_{\ell+p})^\infty$ at some index.

Set $\mathtt{\bm b}=b_0b_1\cdots\in\widetilde{\Sigma}^\infty$ where $b_j=c_{\ell+1+j}$ and let $\zeta_n:=\nu_{\mathtt{\bm b}|n}$ be the node in $\widetilde{\mathsf{I}}^n$. Observe that $\zeta_n$ is by definition the center of the nodal disk $\widetilde{\mathsf{D}}^{\mathtt{\bm b}|n}$ and $\zeta$ is a point inside such disk. Therefore, if $\zeta_n$ is far enough from $\zeta$, then $\omega_n=-\zeta_n+2\zeta$, i.e. the reflection of $\zeta_n$ about $\zeta$, will be outside $\widetilde{\mathsf{D}}^{\mathtt{\bm b}|n}$. We can then find an open disk, $B_n$ centered at $\omega_n$ tangent to $\widetilde{\mathsf{D}}^{\mathtt{\bm b}|n}$ (see Figure~\ref{fig:chainball}).
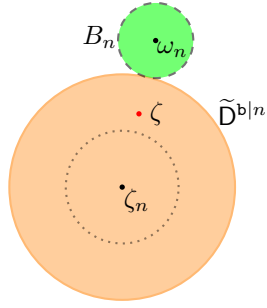
\begin{figure}
\centering
\begin{tikzpicture}
\draw[thick, orange, fill=or, opacity=0.55] (1,0) ellipse (1.5 and 1.5);
\draw[dashed,thick,fill=green!!70,opacity=0.55] (1.445,1.95) ellipse (.5 and .5);
\draw[dotted,thick,black!!70,opacity=0.45] (1,0) ellipse (.75 and .75);
\fill (1,0) circle [radius=1pt] node[below right=-5pt] {\footnotesize $\zeta_n$};
\fill (1.445,1.95) circle [radius=1pt] node[below right=-5pt] {\footnotesize $\omega_n$};
\fill[red] (1.223,0.975) circle [radius=1pt] node[right] {\color{black} \footnotesize $\zeta$};
\node at (2.6,1) {\footnotesize $\widetilde{\mathsf{D}}^{\mathtt{\bm b}|n}$};
\node at (0.7,2) {\footnotesize $B_n$};
\end{tikzpicture}
\caption{Construction of $B_n$, an element in the connected chain of open sets in $\mathbb{C}\setminus\widetilde{\mathsf{A}}_\lambda$. The dotted circle has half the radius of $\widetilde{\mathsf{D}}^{\mathtt{\bm b}|n}$. In order to allow for the existence of $B_n$, the distance between $\zeta$ and $\zeta_n$ must be more than half of the radius of the nodal disk $\widetilde{\mathsf{D}}^{\mathtt{\bm b}|n}$. }
\label{fig:chainball}
\end{figure}

Recall that each nodal disk in $\widetilde{\mathsf{I}}^n$ has a radius of $\left|\lambda^{n+1}\right|(1-\left|\lambda\right|)^{-1}$. Thus, the radius, $r_n$ of $B_n$ is easily found to be \[r_n:=\left|\omega_n-\zeta_n\right|-\frac{\left|\lambda^{n+1}\right|}{1-\left|\lambda\right|}=2\left|\zeta-\zeta_n\right|-\frac{\left|\lambda^{n+1}\right|}{1-\left|\lambda\right|}.\]
\begin{proof}[Proof of Theorem \ref{thm:bdholem}]

\bigskip
We will show that the chain of open disks $\bigcup_{n\geq0}B_n$ is a connected subset of the complement of $\widetilde{\mathsf{A}}_\lambda$, that accumulates at $\zeta$. The claim of the theorem will then follow from Theorem~\ref{thm:asymCKW}{\bf (ii)}, which translates a segment of the chain into a {\em connected}, open set of parameter space that {\em accumulates} at $\lambda$, and is {\em disjoint} from $\mathcal{M}$.

Items {\bf (i)}--{\bf (iii)} below guarantee that the initial $p+1$ disks of the chain exist, are {\em connected} in succession, and are {\em disjoint} from the attractor. The proof concludes by using the periodicity of the itinerary of $\zeta$ to show that the chain consists of an infinite sequence of re-scaled copies of this initial block, {\em accumulating} at $\zeta$. 


Each zero coefficient in $f$ doubles the size of the overlap set (cf. the note in Section~\ref{sect:ovrlp}). Thus, the condition that $\left|O_\lambda\right|$ is finite implies that all zero coefficients of the power series appear in the first $\ell+1$ terms (which do not repeat periodically). The choice of which $\xi\in O_\lambda$ to consider is arbitrary but $\zeta$ is always unique because $\left|\mathcal{F}_\lambda\right|=1$.

We consider only the case $\left|O_\lambda\right|=2$ to simplify notation. By assumption, the coefficients of $f$ are strictly preperiodic, and exactly one of them must be zero (cf. Lemma~\ref{lem:onePOM}). Hence, there exists $0<k\leq\ell$ such that $c_k=0$ which implies $O_\lambda=\{\pm \lambda^k\}=\{\pm\xi\}$. We deduce that the itinerary of $\zeta$ is a word in $\Sigma^\infty$ and, thus, $\zeta\in\mathsf{A}_\lambda\subset\widetilde{\mathsf{A}}_\lambda$. 

Let $\mathtt{\bm a},\overline{\mathtt{\bm a}}\in\Sigma^\infty$ be such that $\pi_\lambda(\mathtt{\bm a})=\pi_\lambda(\overline{\mathtt{\bm a}})=\xi$. Therefore, $a_j=-\overline{a}_j=c_j$ for all $j\neq k$ and $a_k=\overline{a}_k=+1$ or $-1$ (since $c_k=0$). In particular, $\xi=f(\lambda)+a_k\lambda^k$. Denote by $\xi_n$ the nodes with itinerary $\mathtt{\bm a}|n$, i.e. $\xi_n=\nu_{\mathtt{\bm a}|n}=\sum_{j=0}^na_j\lambda^j$. The Taylor polynomial $f_{\ell+1+n}(\lambda)$ for $n\geq0$ can then be written as $\xi_{\ell+1+n}-\xi$.

Observe that the itinerary of $\zeta$ is the (left) shift of $\mathtt{\bm a}$ by $\ell$ terms. Indeed, we claim that 
\[\zeta=\mathfrak{s}_{\mathtt{\bm a}|\ell}^{-1}(\xi)\quad\text{ and }\quad \zeta_n=\mathfrak{s}_{\mathtt{\bm a}|\ell}^{-1}(\xi_{\ell+1+n}).\]
Since $\mathfrak{s}_{a_0a_1}(z)=\mathfrak{s}_{a_0}(\mathfrak{s}_{a_1}(z))=\nu_{a_0a_1}+\lambda^2z$ then $\mathfrak{s}_{a_0a_1}^{-1}(z)=\mathfrak{s}_{a_1}^{-1}(\mathfrak{s}_{a_0}^{-1}(z))=\frac{1}{\lambda^2}(z-\nu_{a_0a_1})$. Consequently,
\begin{eqnarray*}
\mathfrak{s}_{\mathtt{\bm a}|\ell}^{-1}(\xi)=\frac{1}{\lambda^{\ell+1}}\left(\xi-\xi_\ell\right)
&=& -\frac{1}{\lambda^{\ell+1}}f_\ell(\lambda)=\zeta\\
\mathfrak{s}_{\mathtt{\bm a}|\ell}^{-1}(\xi_{\ell+1+n})=\frac{1}{\lambda^{\ell+1}}\left(\xi_{\ell+1+n}-\xi_\ell\right) &=& \frac{1}{\lambda^{\ell+1}}\left(f_{\ell+1+n}(\lambda)-f_\ell(\lambda)\right)=\zeta_n.
\end{eqnarray*}
The centers of the disks $B_n$ were defined in terms of $\zeta_n$ and $\zeta$, but we can now rewrite them in terms of the Taylor polynomials
\[\omega_n=-\zeta_n+2\zeta=-\frac{1}{\lambda^{\ell+1}}\left(f_{\ell+1+n}(\lambda)+f_{\ell}(\lambda)\right).\]
Using the above equations, we also rewrite the radius of $B_n$ in terms of a Taylor polynomial:
\[r_n=\left|\omega_n-\zeta_n\right|-\frac{\left|\lambda^{n+1}\right|}{1-\left|\lambda\right|}=\frac{2}{\left|\lambda^{\ell+1}\right|}\left|f_{\ell+1+n}(\lambda)\right|-\frac{\left|\lambda^{n+1}\right|}{1-\left|\lambda\right|}.\]
We are now practically done. For each $0\leq n\leq p-1$:
\begin{enumerate}[{\bf (i)}]
    \item The disk $B_n$ exists if and only if $r_n>0$, namely
    \[ \frac{2}{\left|\lambda^{\ell+1}\right|}\left|f_{\ell+1+n}(\lambda)\right|-\frac{\left|\lambda^{n+1}\right|}{1-\left|\lambda\right|}>0\iff \left|f_{\ell+1+n}(\lambda)\right|>\frac{1}{2}\frac{\left|\lambda^{\ell+1+n+1}\right|}{1-\left|\lambda\right|},\]
    which is true by assumption.
    \item The disks $B_n,B_{n+1}$ intersect if and only if $r_n+r_{n+1}>\left|\omega_n-\omega_{n+1}\right|=\left|\lambda^{n+1}\right|$, that is
    \begin{eqnarray*} \frac{2}{\left|\lambda^{\ell+1}\right|}\left(\left|f_{\ell+1+n}(\lambda)\right|+\left|f_{\ell+1+n+1}(\lambda)\right|\right)-\frac{\left|\lambda^{n+1}\right|}{1-\left|\lambda\right|}-\frac{\left|\lambda^{n+2}\right|}{1-\left|\lambda\right|}&>&\left|\lambda^{n+1}\right|\\
    \iff \left|f_{\ell+1+n}(\lambda)\right|+\left|f_{\ell+1+n+1}(\lambda)\right|&>&\frac{\left|\lambda^{\ell+1+n+1}\right|}{1-\left|\lambda\right|}
    \end{eqnarray*}
    which is true by assumption.
    \item The disk $B_n$ is tangent to the nodal disk $\widetilde{\mathsf{D}}^{\mathtt{\bm b}|n}$, where $\mathtt{\bm b}$ is the itinerary of $\zeta$, by construction. It is disjoint from the rest of the instar $\widetilde{\mathsf{I}}^n$ if and only if it is disjoint from all other nodal disks; i.e., for every node $\nu_{\mathtt{\bm w}|n}\in\widetilde{\mathsf{I}}^n$ with $\mathtt{\bm w}\in\widetilde{\Sigma}^\infty$ and $\nu_{\mathtt{\bm w}|n}\neq \nu_{\mathtt{\bm b}|n}$ we have $r_n+\left|\lambda^{n+1}\right|(1-\left|\lambda\right|)^{-1}<\left|\omega_n-\nu_{\mathtt{\bm w}|n}\right|$. Again, expanding the expressions for $r_n$ and $\omega_n$, this means
    \begin{eqnarray*}
    \frac{2}{\left|\lambda^{\ell+1}\right|}\left|f_{\ell+1+n}(\lambda)\right|&<&\left|-\frac{1}{\lambda^{\ell+1}}(f_{\ell+1+n}(\lambda)+f_\ell(\lambda))-\nu_{\mathtt{\bm w}|n}\right|\\
    \iff 2\left|f_{\ell+1+n}(\lambda)\right|&<&\left|f_\ell(\lambda)+f_{\ell+1+n}(\lambda)+\lambda^{\ell+1}\nu_{\mathtt{\bm w}|n}\right|\\
    &=&\left|2f_\ell(\lambda)+\lambda^{\ell+1}P(\lambda)\right|
    \end{eqnarray*}
    where $P$ is a polynomial of degree at most $n$ with coefficients taken from the set $\{-2,-1,0,+1,+2\}$. As before, this inequality is true by assumption.
\end{enumerate}
Finally, we claim that $C\lambda^{pm}\in\frac{\lambda^{\ell+1}}{f'(\lambda)}\left(\bigcup_{n\geq0}B_n-\zeta\right)$ for every $C\in\frac{\lambda^{\ell+1}}{f'(\lambda)}\left(\bigcup_{n\geq0}B_n-\zeta\right)$ and $m\geq1$. We will show that, after a translation by $-\zeta$, the chain is forward invariant under $z\mapsto\lambda^pz$. Therefore, showing that the chain accumulates at $\zeta$.\\
Observe that $\omega_n-\zeta=-\lambda^{-\ell-1}f_{\ell+1+n}(\lambda)$ and since $(\lambda^p-1)f_{\ell}(\lambda)=\sum_{j=\ell+1}^{\ell+p}c_j\lambda^j$, then 
\begin{eqnarray*}
\lambda^p(\omega_n-\zeta)&=&-\frac{1}{\lambda^{\ell+1}}\lambda^pf_{\ell+1+n}(\lambda)\\
&=&-\frac{1}{\lambda^{\ell+1}}\lambda^p\left(f_{\ell}(\lambda)+\sum_{j=\ell+1}^{\ell+1+n}c_j\lambda^j\right)\\
&=&-\frac{1}{\lambda^{\ell+1}}\left(\lambda^pf_{\ell}(\lambda)+\lambda^p\sum_{j=\ell+1}^{\ell+1+n}c_j\lambda^j\right)\\
&=&-\frac{1}{\lambda^{\ell+1}}\left(f_\ell(\lambda)+\sum_{j=\ell+1}^{\ell+p}c_j\lambda^j+\sum_{j=\ell+1+p}^{\ell+1+n+p}c_j\lambda^j\right)\\
&=&-\frac{1}{\lambda^{\ell+1}}f_{\ell+1+n+p}(\lambda)=\omega_{n+p}-\zeta
\end{eqnarray*}
where the second to last equality is due to the fact that $c_{\ell+k}=c_{\ell+k+p}$ for every $k\geq1$. 
In particular, $B_{p+j}$ is the rescaled copy of $B_j$, and by part {\bf (ii)} it intersects non-trivially $B_{p-1+j}$ for all $j\geq0$ . Theorem~\ref{thm:asymSol}{\bf (i)} guarantees condition {\bf (iii)} holds for $B_n$ with $n\geq p$.
Hence, the claim follows from Theorem~\ref{thm:asymCKW}{\bf(ii)}.
\end{proof}
\vskip0.1in

The proof of Theorem \ref{thm:bdholem0} is analogous, except we only have to show $\bigcup_{n\geq0}B_n$ is outside $\mathsf{A}_\lambda$. Therefore, in step {\bf (iii)} we need to check that $B_n$ does not intersect the instar $\mathsf{I}^n$, rather than $\widetilde{\mathsf{I}}^n$.

\begin{remark}
The conditions of both Theorems~\ref{thm:bdholem} and \ref{thm:bdholem0} can be weakened as follows: there exists integers $2\leq m\leq p$ and $\{k_1,k_2,\cdots,, k_m \}$ where $0\leq k_1< k_2<\cdots<k_m\leq p-1$ such that for all $1\leq j\leq m$
\begin{enumerate}[{\bf (i)}]
\item $ \left|f_{\ell+1+k_j}(\lambda)\right| >\frac{1}{2}\frac{\left|\lambda^{\ell+1+k_j+1}\right|}{1-\left|\lambda\right|} $
\item $ \left|f_{\ell+1+k_j}(\lambda)\right|+\left|f_{\ell+1+k_{j+1}}(\lambda)\right|-\frac{1}{2}\left|f_{\ell+1+k_j}(\lambda)-f_{\ell+1+k_{j+1}}(\lambda)\right|>\frac{1}{2}\frac{\left|\lambda^{\ell+2+k_j}\right|+\left|\lambda^{\ell+2+k_{j+1}}\right|}{1-\left|\lambda\right|}$
\item Let $Q$ be the solution of  $f_\ell(z)+z^{\ell+1}Q(z)=f_{\ell+1+k_j}(z)$. Then for any polynomial $P$ of degree less or equal to $k_j$ with coefficients in $\{-1,0,+1\}$, with $P\neq Q$ it holds that $\left|f_{\ell+1+k_j}(\lambda)\right|<\left|f_\ell(\lambda)+\lambda^{\ell+1}P(\lambda)\right|$.
\end{enumerate}
These conditions allow the possibility of intersection between two non-consecutive chain disks.
\end{remark}


\section{Examples: Landmark Points}
\label{sect:ThmExamples}
In \cite{Sol}, Solomyak discussed six sample parameters that satisfy $\left|\mathcal{F}_{\lambda_j}\right|=1$, and named them {\em landmark points}. As a note of caution, we changed Solomyak's nomenclature to tie in with our exposition, so that our $\lambda_1,\lambda_2,\lambda_3,\lambda_4,\lambda_5,\lambda_6$ correspond to his $\lambda_3,\lambda_4,\lambda_5,\lambda_6,\lambda_1,\lambda_2$. In this section we will prove that $\lambda_j$ for $j=1,\ldots,5$ satisfy the conditions of Theorems~\ref{thm:bdholem} and~\ref{thm:bdholem0} and are therefore accessible.

\subsection{Period One} 
The landmark points $\lambda_j$ for $j=1,\ldots,4$ are all inside the following sector:
\[{\bm S}:=\left\{z\in\mathbb{D}~\bigg|~\frac{\sqrt{5}-1}{2}<\left|z\right|<\frac{2}{3}\text{ and }0<\arg(z)<\frac{5\pi}{32}\right\} \]
and have itineraries of various preperiods, but all with period $1$. By Proposition~\ref{prop:genlandmark} they are all accessible. Later in Section~\ref{subsect:ppm} we present a method to circumvent the lengthier computations that would be required to establish the inequalities for $\lambda_5$ (which has period $3$).
\begin{proposition}
\label{prop:genlandmark}
Let $\lambda\in{\bm S}$ and assume $\mathcal{F}_\lambda=\{f\}$ where 
$f$ is of rational type $(\ell,1)$, that is:
\[f(z)=\sum_{j=0}^\ell c_jz^j+\frac{z^{\ell+1}}{1-z}.\]
Then $\lambda$ satisfies the conditions of Theorem~\ref{thm:bdholem} (and therefore $\lambda\in\partial\mathcal{M}$ is accessible).
\end{proposition}
The proof will use the following properties of ${\bm S}$:
\begin{lemma}
\label{lem:usefulInequal}
For all $\lambda\in{\bm S}$ the following holds: 
\begin{enumerate}[{\bf (a.)}]
    \item $1-\left|\lambda\right|>\frac{1}{2}\left|1-\lambda\right|$;
    \item $1-\left|\lambda\right|^2>\left|1-\lambda\right|$;
    \item $\left|\lambda\right|<\left|2-\lambda\right|$;
    \item $2\left|\lambda\right|<\left|3-\lambda\right|$;
    \item $2\left|\lambda\right|<\left|1+\lambda\right|$.
\end{enumerate}
\end{lemma}
\begin{proof}
From the law of cosines we obtain
    \begin{align*}
    \frac{5}{9}&>\sqrt{1^2+\left(\frac{\sqrt{5}-1}{2}\right)^2-2\left(\frac{\sqrt{5}-1}{2}\right)\cos\left(\frac{5\pi}{32}\right)}\\
    &>\left|1-\lambda\right|>\sqrt{1^2+\left(\tfrac{2}{3}\right)^2-2\left(\tfrac{2}{3}\right)\cos(0)}=\frac{1}{3}.
    \end{align*}
    It follows that 
    \[1-\left|\lambda\right|>1-\tfrac{2}{3}>\tfrac{1}{2}\cdot \tfrac{5}{9}>\tfrac{1}{2}\left|1-\lambda\right|\]
    and
    \[1-\left|\lambda\right|^2>1-\tfrac{4}{9}>\left|1-\lambda\right|,\]
    which gives {\bf(a.)} and {\bf (b.)}. 
    
    Similarly,
    \begin{align*}
       \left|2-\lambda\right|&>\sqrt{2^2+\left(\frac{2}{3}\right)^2-2\left(\frac{2}{3}\right)2\cos(0)}=\frac{4}{3}>2\left|\lambda\right|>\left|\lambda\right|\\
       \left|3-\lambda\right|&>\sqrt{3^2+\left(\frac{2}{3}\right)^2-2\left(\frac{2}{3}\right)3\cos(0)}=\frac{7}{3}>2\left|\lambda\right|,
    \end{align*}
    which gives {\bf (c.)} and {\bf (d.)}.
    
    Finally, 
    \[
    \left|1+\lambda\right|>\sqrt{1^2+\left(\frac{\sqrt{5}-1}{2}\right)^2+2\left(\frac{\sqrt{5}-1}{2}\right)\cos\left(\frac{5\pi}{32}\right)}>1.572>\frac{4}{3}>2\left|\lambda\right|,
    \]
    giving {\bf (e.)} and concluding the proof of the lemma.
\end{proof}
Armed with these inequalities, we proceed to prove the Proposition:
\begin{proof}[Proof of Proposition~\ref{prop:genlandmark}]
Since $\lambda$ is the root of the power series $f$, we can write Taylor polynomials as follows
\begin{align*}
    f_\ell(\lambda)&=\sum_{j=0}^\ell c_j\lambda^j=-\frac{\lambda^{\ell+1}}{1-\lambda},\\
    f_{\ell+1}(\lambda)&=\sum_{j=0}^\ell c_j\lambda^j+\lambda^{\ell+1}=-\frac{\lambda^{\ell+2}}{1-\lambda},\\
    f_{\ell+2}(\lambda)&=\sum_{j=0}^\ell c_j\lambda^j+\lambda^{\ell+1}+\lambda^{\ell+2}=-\frac{\lambda^{\ell+3}}{1-\lambda}.
\end{align*}
Condition {\bf(i)} in Theorem~\ref{thm:bdholem} is satisfied since
\[
\left|f_{\ell+1}(\lambda)\right|>\frac{1}{2}\,\frac{\left|\lambda^{\ell+2}\right|}{1-\left|\lambda\right|}\iff \left|\frac{1}{1-\lambda}\right|>\frac{1}{2}\,\frac{1}{1-\left|\lambda\right|}\iff1-\left|\lambda\right|>\frac{1}{2}\left|1-\lambda\right|
\]
holds by Lemma~\ref{lem:usefulInequal} part {\bf (a.)}.

\noindent Condition {\bf(ii)} in Theorem~\ref{thm:bdholem} is satisfied since
\begin{align*}
\left|f_{\ell+1}(\lambda)\right|+\left|f_{\ell+2}(\lambda)\right|>\frac{\left|\lambda^{\ell+2}\right|}{1-\left|\lambda\right|}&\iff \left|\frac{1}{1-\lambda}\right|+\left|\frac{\lambda}{1-\lambda}\right|>\frac{1}{1-\left|\lambda\right|}\\&\iff1-\left|\lambda\right|^2>\left|1-\lambda\right|
\end{align*}
holds by Lemma~\ref{lem:usefulInequal} part {\bf (b.)}.

\noindent Condition {\bf(iii)} in Theorem~\ref{thm:bdholem} has four cases since the polynomial $P$ can only be either $-2,-1,0,$ or $1$. The case $P(z)=-2$ is satisfied because
\begin{align*}
\left|2f_{\ell+1}(\lambda)\right|<\left|2f_{\ell}(\lambda)+\lambda^{\ell+1}(-2)\right|&\iff \left|\frac{-2\lambda^{\ell+2}}{1-\lambda}\right|<\left|\frac{-4\lambda^{\ell+1}+2\lambda^{\ell+2}}{1-\lambda}\right|\\
&\iff \left|\lambda\right|<\left|2-\lambda\right|
\end{align*}
holds by Lemma~\ref{lem:usefulInequal} part {\bf (c.)}.\\
The case $P(z)=-1$ is satisfied because
\begin{align*}
\left|2f_{\ell+1}(\lambda)\right|<\left|2f_{\ell}(\lambda)+\lambda^{\ell+1}(-1)\right|&\iff \left|\frac{-2\lambda^{\ell+2}}{1-\lambda}\right|<\left|\frac{-3\lambda^{\ell+1}+\lambda^{\ell+2}}{1-\lambda}\right|\\
&\iff \left|2\lambda\right|<\left|3-\lambda\right|
\end{align*}
holds by Lemma~\ref{lem:usefulInequal} part {\bf (d.)}.\\
The case $P(z)=0$ is trivial since
\begin{align*}
\left|2f_{\ell+1}(\lambda)\right|<\left|2f_{\ell}(\lambda)+\lambda^{\ell+1}(0)\right|&\iff \left|\frac{-2\lambda^{\ell+2}}{1-\lambda}\right|<\left|\frac{-2\lambda^{\ell+1}}{1-\lambda}\right|
\iff \left|\lambda\right|<1.
\end{align*}
The case $P(z)=1$ is satisfied because
\begin{align*}
\left|2f_{\ell+1}(\lambda)\right|<\left|2f_{\ell}(\lambda)+\lambda^{\ell+1}(1)\right|&\iff \left|\frac{-2\lambda^{\ell+2}}{1-\lambda}\right|<\left|\frac{-\lambda^{\ell+1}-\lambda^{\ell+2}}{1-\lambda}\right|\\
&\iff \left|2\lambda\right|<\left|1+\lambda\right|
\end{align*}
holds by Lemma~\ref{lem:usefulInequal} part {\bf (e.)}.
\end{proof}
\begin{figure}[!htb]
    \centering
    \includegraphics[scale=0.47]{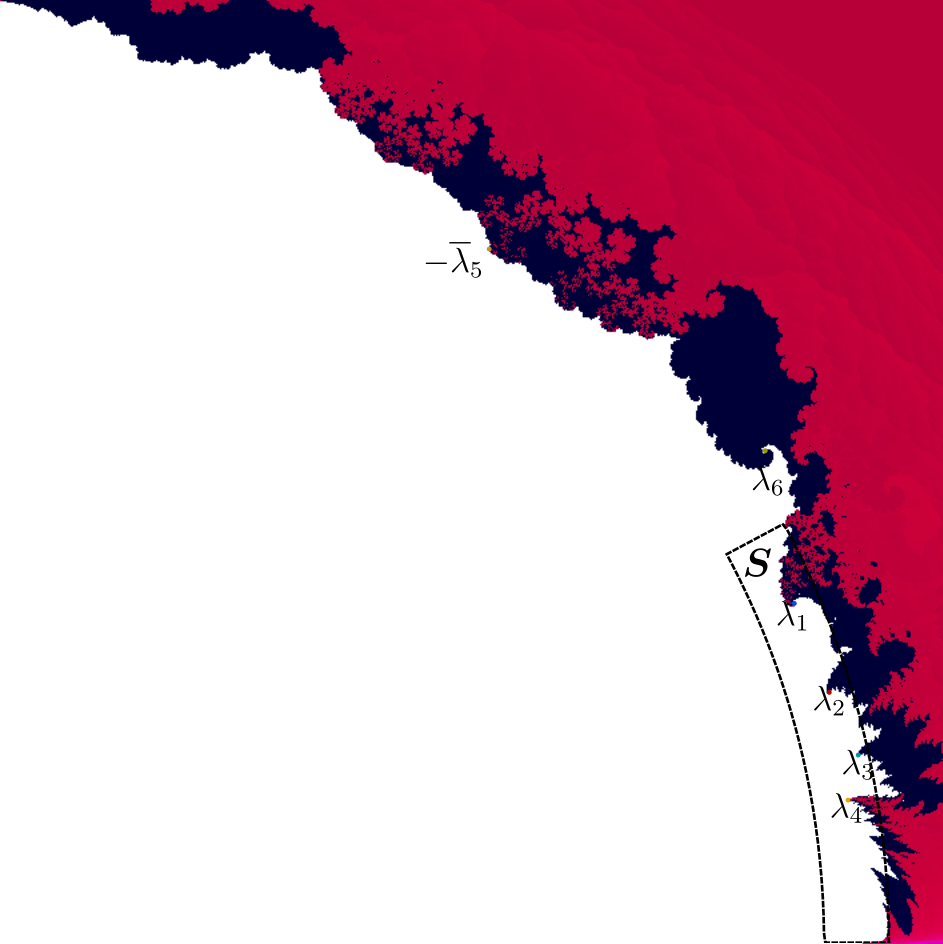}
    \caption{The window $[0,0.708]^2$ of parameter space, showing the sector ${\bm S}$ intersecting $\mathcal{M}$ and $\mathcal{M}_0$. The landmark points $\lambda_j$ for $j=1,2,3,4,5,6$ are marked (we display $-\overline{\lambda_5}$ instead of $\lambda_5$ simply to make a picture with better resolution; see also the footnote to Example 5).}
    \label{fig:Sregion}
\end{figure}


In the following examples recall that the uniqueness of the power series was proved by Solomyak in \cite{Sol}.
\begin{example}
Let $\mathtt{\bm c}=c_0c_1\cdots=+--(+)^\infty\in\Sigma^\infty$ and let $f$ be the corresponding power series of rational type $(2,1)$, that is $f(z):=\sum_{k=0}^\infty c_kz^k=\dfrac{1-2z+2z^3}{1-z}$. The associated parameter in $\mathcal{M}\cap{\bm S}$ is $\lambda_1\approx 0.5957439 + 0.2544259{\rm i}$.
It follows from Proposition~\ref{prop:genlandmark} that the parameter $\lambda_1$, is an accessible point of $\partial\mathcal{M}$ (see Figure~\ref{fig:laChain}). In fact, by Corollary~\ref{cor:sharedBd}, we can conclude that $\lambda_1\in\partial\mathcal{M}\cap\partial\mathcal{M}_0$.
\end{example}

\begin{example}
Let $\mathtt{\bm c}=c_0c_1\cdots=+--\mathtt{O}(+)^\infty\in\widetilde{\Sigma}^\infty$ and let $f$ be the corresponding power series of rational type $(3,1)$, that is $f(z):=\sum_{k=0}^\infty c_kz^k=\dfrac{1-2z+z^3+z^4}{1-z}$. The associated parameter in $\mathcal{M}\cap{\bm S}$ is $\lambda_2\approx 0.6219644+0.1877304{\rm i}$.
It follows from Proposition~\ref{prop:genlandmark} that the parameter $\lambda_2$, is an accessible point of $\partial\mathcal{M}$ (see Figure~\ref{fig:laChain}). 

\end{example}

\begin{example}
Let $\mathtt{\bm c}=c_0c_1\cdots=+--\mathtt{O}\mathtt{O}(+)^\infty\in\widetilde{\Sigma}^\infty$ and let $f$ be the corresponding power series of rational type $(4,1)$, that is $f(z):=\sum_{k=0}^\infty c_kz^k=\dfrac{1-2z+z^3+z^5}{1-z}$. The associated parameter in $\mathcal{M}\cap{\bm S}$ is $\lambda_3\approx 0.643703+0.140749{\rm i}$. 
It follows from Proposition~\ref{prop:genlandmark} that the parameter $\lambda_3$, is an accessible point of $\partial\mathcal{M}$ (see Figure~\ref{fig:laChain}).

\end{example}

\begin{example}
Let $\mathtt{\bm c}=c_0c_1\cdots=+---(+)^\infty\in\Sigma^\infty$ and let $f$ be the corresponding power series of rational type $(3,1)$, that is $f(z):=\sum_{k=0}^\infty c_kz^k=\dfrac{1-2z+2z^4}{1-z}$. The associated parameter in $\mathcal{M}\cap{\bm S}$ is $\lambda_4\approx 0.63601 + 0.106924{\rm i}$. 
It follows from Proposition~\ref{prop:genlandmark} that the parameter $\lambda_4$, is an accessible point of $\partial\mathcal{M}$ (see Figure~\ref{fig:laChain}). In fact, by Corollary~\ref{cor:sharedBd}, we can conclude that $\lambda_4\in\partial\mathcal{M}\cap\partial\mathcal{M}_0$.
\end{example}

\begin{note}
The overlap set has different cardinalities in each of these examples. In particular, in Example 1 and 4 (where the overlap is a singleton) the parameters belong to both $\mathcal{M}$ and $\mathcal{M}_0$, whereas in Example 2 and 3 (where the overlap consists of $2$ or $4$ points) the parameters belong to $\mathcal{M}$, but not $\mathcal{M}_0$.
\end{note}

\begin{figure}[!htb]
    \centering
    \includegraphics[scale=0.23, trim= 0 110 0 200,clip]{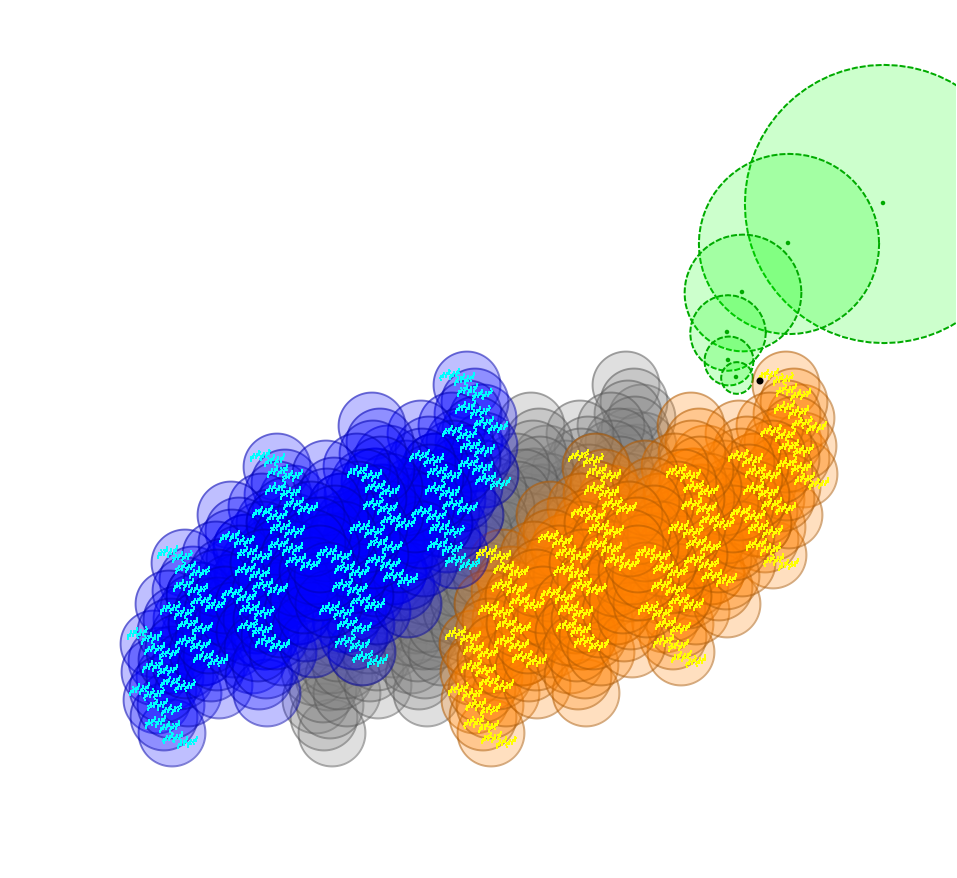}%
    \includegraphics[scale=0.23, trim= 0 135 0 190,clip]{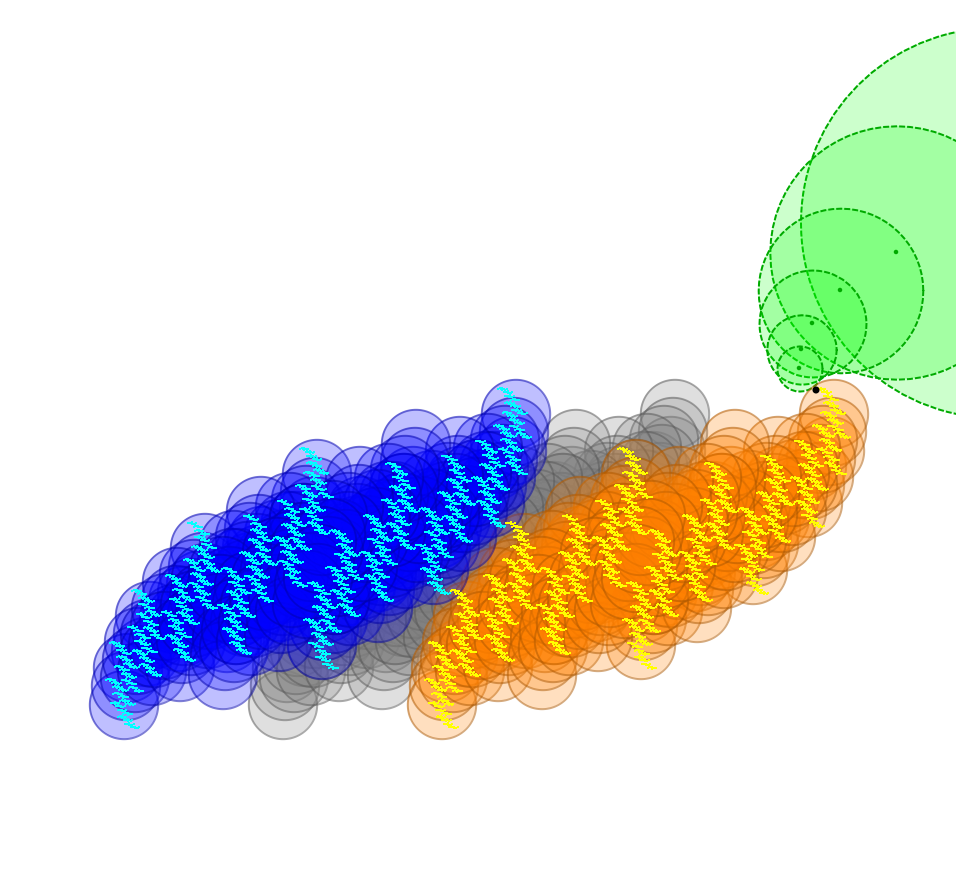}\\
    \includegraphics[scale=0.23, trim= 0 135 0 190,clip]{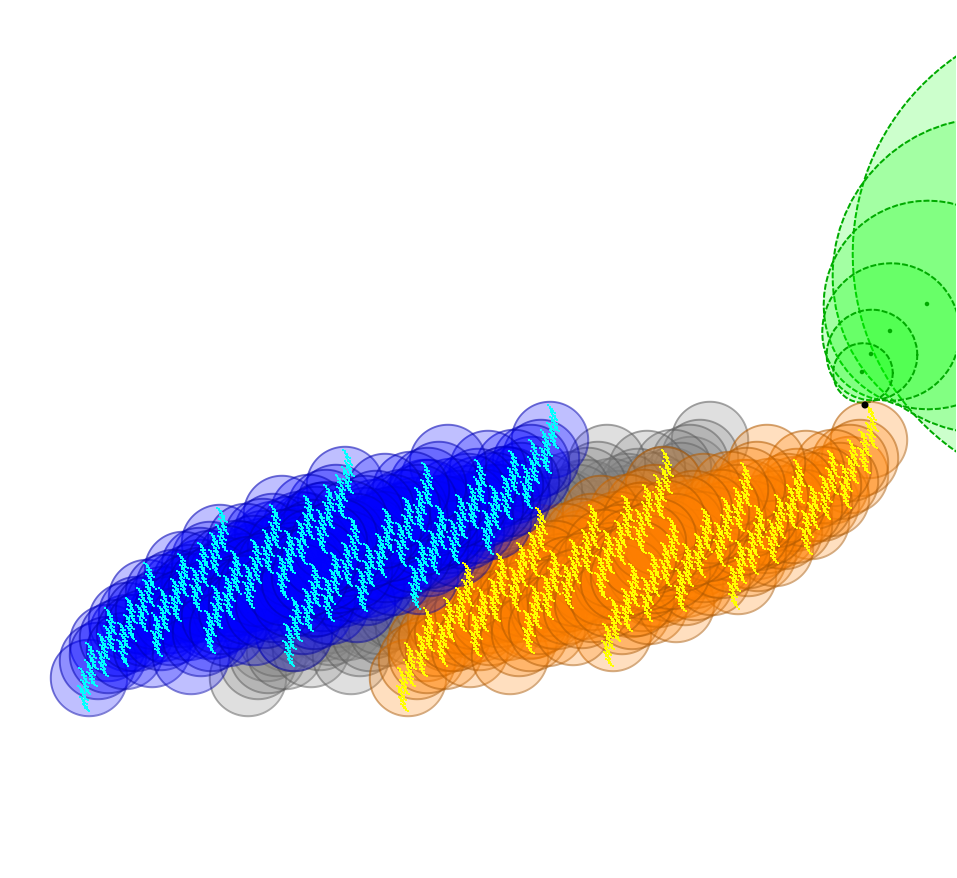}%
    \includegraphics[scale=0.23, trim= 0 135 0 190,clip]{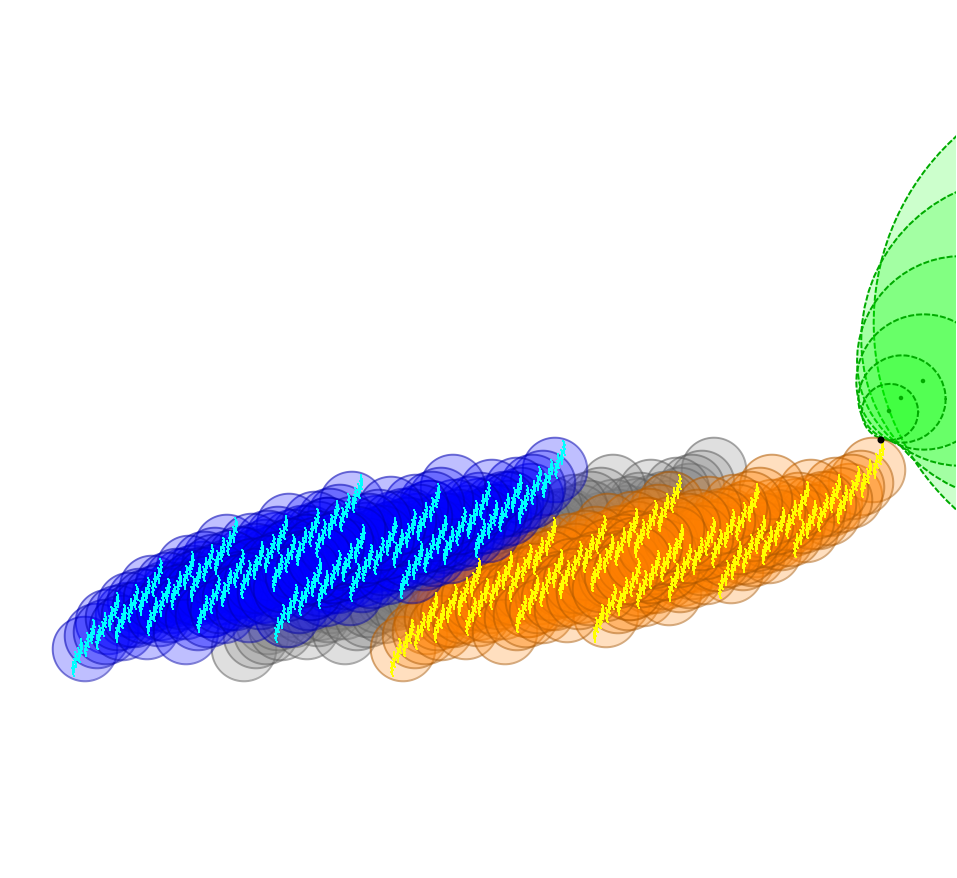}
    \caption{ The attractors $\mathsf{A}_{\lambda_j}$ for $j=1,2,3,4$ (in the order $\begin{smallmatrix}1&2\\3&4\end{smallmatrix}$) inside their respective instars $\widetilde{\mathsf{I}}^{5}$. The green disks in the top-right corners of each plot are part of the chain constructed in Theorem~\ref{thm:bdholem}.}
    \label{fig:laChain}
\end{figure}

\subsection{Higher period}
\label{subsect:ppm}
Here we consider the landmark point $\lambda_5$ whose associated itinerary has period $3$. Observe that in this case the number of inequalities to be checked before applying Theorem~\ref{thm:bdholem} becomes quite large. To overcome this difficulty we rely on the argument and modulus of $\lambda_5$, and the geometric structure of the instar at each level. We will construct each of the disks of the connected chain that lies in the complement of $\widetilde{\mathsf{A}}_{\lambda_5}$.

\begin{example} Let $\mathtt{\bm c}=c_0c_1\cdots=+(++-)^\infty\in\Sigma^\infty$ and consider the associated power series $f(z):=\sum_{k=0}^\infty c_kz^k=\dfrac{1+z+z^2-2z^3}{1-z^3}$ of rational type $(0,3)$. Solomyak proved in \cite{Sol} that\footnote{Solomyak worked with the word $+(-+++--)^\infty$ which gives the power series $\frac{1-z+z^2+2z^3}{1+z^3}$; in other words, he looked at $f(-z)$. The symmetry of $\mathcal{M}$ allows us to consider $f(z)$ instead.} $f$ is the unique power series for which $\lambda_5\approx -0.366+0.520i$ is a root. Using the notation of Theorem~\ref{thm:bdholem} we have $\ell=0$, $p=3$, $\xi=0$, $\xi_n=f_{n}(\lambda_5)$, and $\zeta=-\frac{1}{\lambda_5}\in\mathsf{A}_{\lambda_5}\subset\widetilde{\mathsf{A}}_{\lambda_5}$ with itinerary $\mathtt{\bm b}=(++-)^\infty$. 

We begin by proving that the hypothesis {\bf (i)} of Theorem~\ref{thm:bdholem} is satisfied:
\begin{lemma}
\label{lem:existOneThird}
    Let $\lambda_5$ and $f(z)$ be as above. Then for every $0\leq n\leq p-1$ 
    \[ 2\left|f_{n}(\lambda_5)\right|>\frac{\left|\lambda_5^{n+1}\right|}{1-\left|\lambda_5\right|}.\]
\end{lemma}
\begin{proof}First notice that 
    \begin{enumerate}[(i)]
        \item $\frac{\sqrt{5}-1}{2}<\left|\lambda_5\right|<\frac{2}{3}$ and 
        \item $\frac{2}{3}\pi<\arg(\lambda_5)<\frac{23}{32}\pi$.
    \end{enumerate}
    Then, from (i) we get \[2\left|f_{0}(\lambda_5)\right|=2\left|1\right|>\frac{\left|\lambda_5\right|}{1-\left|\lambda_5\right|}.\]

    From (ii) and the Law of cosines we get
    \begin{eqnarray*}
    2\left|f_{1}(\lambda_5)\right|=2\left|1+\lambda_5\right|&>&2\sqrt{1+\left|\lambda_5\right|^2-2\left|\lambda_5\right|\cos\left(\frac{9\pi}{32}\right)}\\
    &>&2\sqrt{1+\frac{(\sqrt{5}-1)^2}{4}-(\sqrt{5}-1)\cos\left(\frac{9\pi}{32}\right)}\\
    &>&2~\frac{1}{\sqrt{2}}>2\left|\lambda_5\right|>\frac{\left|\lambda_5^2\right|}{1-\left|\lambda_5\right|}.
    \end{eqnarray*}
    Since $f(\lambda_5)=0$ then $f_{2}(\lambda_5)=1+\lambda_5+\lambda_5^2=2\lambda_5^3$ and consequently 
    \[2\left|f_{2}(\lambda_5)\right|=2\left|2\lambda_5^3\right|>\frac{\left|\lambda_5^3\right|}{1-\left|\lambda_5\right|}\]
    because $4>(1-\left|\lambda_5\right|)^{-1}$.
\end{proof}

Lemma~\ref{lem:existOneThird} guarantees that the disks $B_n$ in the chain exists for every $n\geq0$. We can show now that $B_n\cap B_{n+1}\neq\emptyset$ and that $B_n\subset \mathbb{C}\setminus\widetilde{\mathsf{I}}^n$, namely the remaining two hypothesis of Theorem~\ref{thm:bdholem}. Recall that each disk $B_n$ is centered at $\omega_n$ and has radius $r_n$:
\[
\omega_n=-\frac{1}{\lambda_5}(f_{n+1}(\lambda_5)+f_{0}(\lambda_5)),\qquad
r_n=\frac{2}{\left|\lambda_5\right|}\left|f_{n+1}(\lambda_5)\right|-\frac{\left|\lambda_5^{n+1}\right|}{1-\left|\lambda_5\right|}.
\]
\begin{lemma}
\label{lem:outOneThird}
For $\lambda_5\approx-0.366+0.520{\rm i}$, the set $\bigcup_{n\geq0}B_n$ is connected and lies in the complement of $\widetilde{\mathsf{A}}_{\lambda_5}$.
\end{lemma}
\begin{proof}
 By the self-similarity of the attractor $\mathsf{A}_{\lambda_5}$ at $\zeta$ it is enough to prove that $\cup_{0\leq n\leq2} B_n$ is connected and lies outside the instar $\widetilde{\mathsf{I}}^2$. We will first show that $B_n$ lies outside the instar $\widetilde{\mathsf{I}}^n$ and then prove that $B_n\cap B_{n+1}\neq\emptyset$ for each $n=0,1,2$.
\begin{itemize}
    \item[$n=0$:] The instar $\widetilde{\mathsf{I}}^0$ is the union of the nodal disks $\mathsf{D}^+,\mathsf{D}^\mathtt{O}$, and $\mathsf{D}^-$. We already know that $B_0$ is outside $\mathsf{D}^+$, we want to show that the distances between $\omega_0$ and the nodes $\nu_{\mathtt{\bm 0}}=0$ and $\nu_{-}=-1$ is larger than the sum of the radii of $B_0$ and the nodal disk of level $0$. In other words, we need that $r_0+\left|\lambda_5\right|(1-\left|\lambda_5\right|)^{-1}<\left|\omega_0-0\right|$ and $r_0+\left|\lambda_5\right|(1-\left|\lambda_5\right|)^{-1}<\left|\omega_0-(-1)\right|$.
    Using that $2^{-1}(\sqrt{5}-1)<\left|\lambda_5\right|<2/3$ and $\arg(\lambda_5)>2\pi/3$, the Law of cosines gives
    \[
    \left|1+\lambda_5\right|<\sqrt{1+\left|\lambda_5\right|^2-2\left|\lambda_5\right|\cos\left(\frac{\pi}{3}\right)}<\sqrt{1+\frac{4}{9}-\frac{4}{3}\cos\left(\frac{\pi}{3}\right)}<1
    \]
    so 
    \[
    \left|\omega_0-(-1)\right|=2\left|\frac{1}{\lambda_5}\right|>2\left|\frac{1+\lambda_5}{\lambda_5}\right|=r_0+\frac{\left|\lambda_5\right|}{1-\left|\lambda_5\right|}.
    \]
    Moreover, using the better estimate $\frac{11}{16}\pi<\arg(\lambda_5)<\frac{45}{64}\pi$, 
    \[
    2\left|1+\lambda_5\right|<2\sqrt{1+\frac{4}{9}-\frac{4}{3}\cos\left(\frac{5\pi}{16}\right)}<\sqrt{2^2+\frac{4}{9}-2\;\frac{4}{3}\cos\left(\frac{19\pi}{64}\right)}
    <\left|2+\lambda_5\right|\]
    implying that
    \[\left|\omega_0-0\right|=\left|\frac{2+\lambda_5}{\lambda_5}\right|>2\left|\frac{1+\lambda_5}{\lambda_5}\right|=r_0+\frac{\left|\lambda_5\right|}{1-\left|\lambda_5\right|}.\]

    \item[$n=1$:]The instar $\widetilde{\mathsf{I}}^1$ is the union of the nodal disks $\widetilde{\mathsf{D}}^{w_0+},\widetilde{\mathsf{D}}^{w_0\mathtt{O}}$, and $\widetilde{\mathsf{D}}^{w_0-}$ where $w_0\in\{+,\mathtt{O},-\}$. Instead of checking that eight more inequalities are satisfied, we use the fact that $\arg(\lambda_5)\in(2\pi/3,23\pi/32)$ to show $0<{\rm Re}(\zeta_1)<{\rm Re}(\zeta)$. Given that $\zeta_1=\frac{1}{\lambda_5}(f_{1+1}(\lambda_5)-f_0(\lambda_5))=1+\lambda_5$ we need to check \[-\frac{1}{\left|\lambda_5\right|}\cos(\arg(\lambda_5))>1+\left|\lambda_5\right|\cos(\arg(\lambda_5))\iff -\cos(\arg(\lambda_5))> \frac{\left|\lambda_5\right|}{1+\left|\lambda_5^2\right|}.\]
    Indeed, $-\cos(\arg(\lambda_5))\in(0.55,0.64)$ and $\tfrac{\left|\lambda_5\right|}{1+\left|\lambda_5^2\right|}\in(0.44,0.47)$. Moreover, we have ${\rm Re}(\zeta_1)>0$ because $\cos(\arg(\lambda_5))>-3/2>-\left|\lambda_5\right|^{-1}$. Consequently, $\omega_1$ is in the first quadrant with ${\rm Re}(\omega_1)>{\rm Re}(\zeta_1)$. It follows that $B_1$ has a chance to intersect only the disks $\widetilde{\mathsf{D}}^{+\mathtt{O}}$ and $\widetilde{\mathsf{D}}^{+-}$. However, in Lemma~\ref{lem:existOneThird} we showed that $\left|1+\lambda_5\right|>2^{-1/2}$, therefore, since $2\left|\lambda_5^3\right|<2^{-1/2}$ we have
    \begin{eqnarray*}
    \left|\omega_1-(1-\lambda_5)\right|&=&2\left|\frac{1+\lambda_5}{\lambda_5}\right|>4\left|\lambda_5^2\right|=r_1+\frac{\left|\lambda_5^2\right|}{1-\left|\lambda_5\right|}.
    \end{eqnarray*}
    so $B_1\cap\widetilde{\mathsf{D}}^{+-}=\emptyset$.
    
    Finally, $B_1$ does not intersect $\widetilde{\mathsf{D}}^{+\mathtt{O}}$ because
    \begin{eqnarray*}
    \left|\omega_1-(1+0\cdot\lambda_5)\right|&=&\left|\frac{-\lambda_5^2+4\lambda_5^3}{\lambda_5}\right|=\left|\lambda_5\right|\left|1-4\lambda_5\right|>\left|\lambda_5\right|\big(1+4\left|\lambda_5\right|\cos(\pi/4)\big)\\
    &>&\left|\lambda_5\right|\frac{8}{3}>4\left|\lambda_5^2\right|=r_1+\frac{\left|\lambda_5^2\right|}{1-\left|\lambda_5\right|}.
    \end{eqnarray*}
    In the first equality we used the fact that $1+\lambda_5=-\lambda_5^2+2\lambda_5^3$ and in the second line that $1+\sqrt{2}\big(\sqrt{
    5}-1\big)>8/3>4\left|\lambda_5\right|$.
    \item[$n=2$:] In this case we prove that $B_2\subset B_1$ which implies that $B_2$ is outside $\widetilde{\mathsf{I}}^2$ because of the containment $\widetilde{\mathsf{I}}^2\subset\widetilde{\mathsf{I}}^1$. Firstly, we have $\left|\omega_1-\omega_2\right|=\left|\lambda_5^2\right|$ and that the radius of $B_1$ is $r_1=4\left|\lambda_5^2\right|-\frac{\left|\lambda_5^2\right|}{1-\left|\lambda_5\right|}>\left|\lambda_5^2\right|$ because $3>(1-\left|\lambda_5\right|)^{-1}$. Secondly, the radius of $B_2$ is  $r_2=2\left|\lambda_5^2\right|-\frac{\left|\lambda_5^3\right|}{1-\left|\lambda_5\right|}$, so for the containment to hold we must have
    \[4\left|\lambda_5^2\right|-\dfrac{\left|\lambda_5^2\right|}{1-\left|\lambda_5\right|}>2\left(2\left|\lambda_5^2\right|-\frac{\left|\lambda_5^3\right|}{1-\left|\lambda_5\right|}\right),\]
    which holds since $2\left|\lambda_5\right|>1$.
\end{itemize}

We have shown that $B_0,B_1$, and $B_2$ do not intersect their respective instar and, therefore, their union lies outside the instar $\widetilde{\mathsf{I}}^2\supset\widetilde{\mathsf{A}}_{\lambda_5}$. We have also proved that $B_2\subset B_1$ so it remains to show that $B_1\cap B_0\neq\emptyset$ and $B_2\cap B_3\neq\emptyset$.

We will prove that $\omega_1\in B_0$, i.e. $\left|\omega_0-\omega_1\right|=\left|\lambda_5\right|<r_0$: recall that $r_0=2\left|\frac{1+\lambda_5}{\lambda_5}\right|-\frac{\left|\lambda_5\right|}{1-\left|\lambda_5\right|}$ and $0.63<\left|\lambda_5\right|<0.64$ then
\[
2\left|1-2\lambda_5\right|>2\big(1+2(0.63)\cos(\pi/4)\big)>\frac{1}{1-(0.64)}+1>\frac{1}{1-\left|\lambda_5\right|}+1
\] 
which implies 
\[2\left|\frac{1+\lambda_5}{\lambda_5}\right|=2\left|\frac{-\lambda_5^2+2\lambda_5^3}{\lambda_5}\right|=2\left|\lambda_5\right|\left|1-2\lambda_5\right|>\left|\lambda_5\right|\left(\frac{1}{1-\left|\lambda_5\right|}+1\right).\]

Finally, we want to show that $\left|\omega_2-\omega_3\right|\left|\lambda_5^3\right|=\left|\lambda_5^3\right|<r_2+r_3$. 
By definition $r_3=2\left|\lambda_5^2\right|\left|1+\lambda_5\right|-\frac{\left|\lambda_5^4\right|}{1-\left|\lambda_5\right|}$ and since 
\begin{eqnarray*}
\left|\lambda_5\right|+\frac{\left|\lambda_5^2\right|}{1-\left|\lambda_5\right|}< 0.639+\frac{(0.639)^2}{1-0.639}&<&1+\sqrt{1+\frac{(\sqrt{5}-1)^2}{4}-(\sqrt{5}-1)\cos\left(\frac{9\pi}{32}\right)}\\
&<&1+\left|1+\lambda_5\right|
\end{eqnarray*}
then
\begin{eqnarray*}
\left|\lambda_5\right|&<&1+\left|1+\lambda_5\right|-\frac{\left|\lambda_5^2\right|}{1-\left|\lambda_5\right|}\\
\iff\left|\lambda_5\right|&<&2+2\left|1+\lambda_5\right|-\left|\lambda_5\right|-\frac{2\left|\lambda_5^2\right|}{1-\left|\lambda_5\right|}\\
&=&2-\frac{\left|\lambda_5\right|}{1-\left|\lambda_5\right|}+2\left|1+\lambda_5\right|-\frac{\left|\lambda_5^2\right|}{1-\left|\lambda_5\right|}\\
\iff\left|\lambda_5^3\right|&<&r_2+r_3.
\end{eqnarray*}
The equality in the third line holds because $\frac{\left|\lambda_5\right|}{1-\left|\lambda_5\right|}=\left|\lambda_5\right|+\frac{\left|\lambda_5^2\right|}{1-\left|\lambda_5\right|}$.

\end{proof}

By Lemma~\ref{lem:outOneThird} the chain $\bigcup_{n\geq0}B_n$ is connected and lies in the complement of $\widetilde{\mathsf{A}}_{\lambda_5}$ (see Figure~\ref{fig:la5chai}). Then, by Theorem~\ref{thm:bdholem}, $\lambda_5\in\partial\mathcal{M}$ is an accessible point. Because $f$ is unique and has no zero coefficients, by Corollary~\ref{cor:sharedBd}, $\lambda_5$ is an accessible point of $\partial\mathcal{M}\cap\partial\mathcal{M}_0$.
\begin{figure}
    \centering
    \includegraphics[scale=0.35, trim=0 100 0 180,clip]{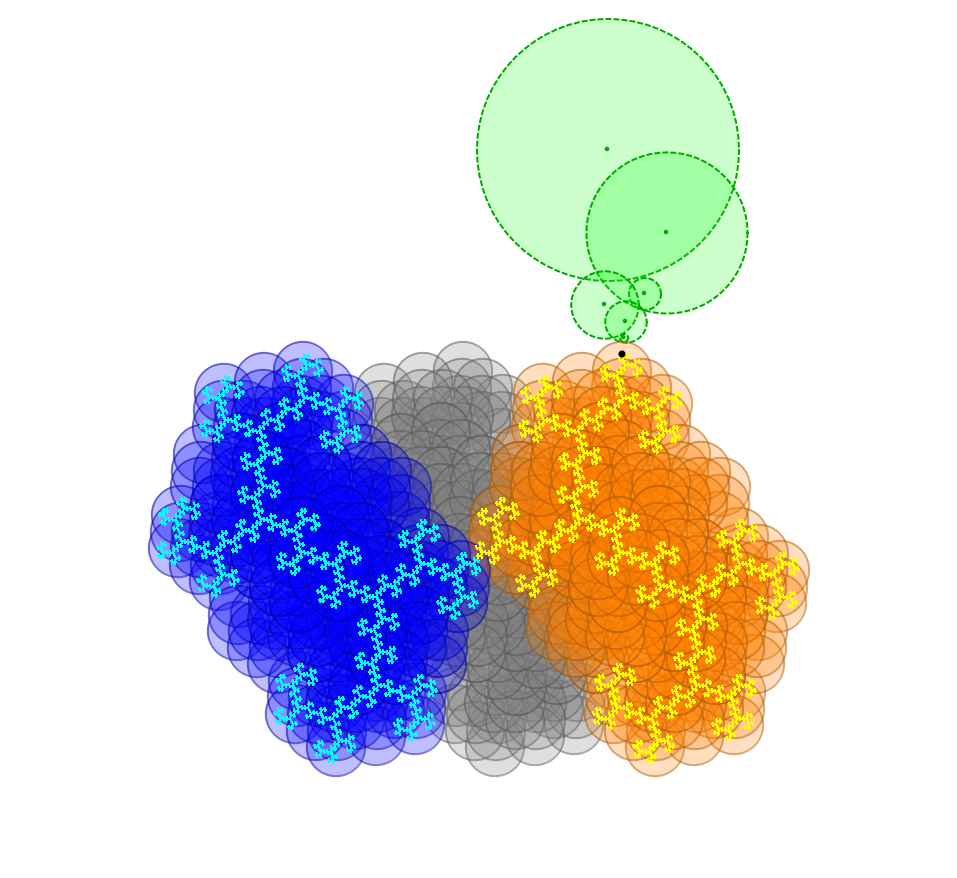}
    \caption{The attractor $\mathsf{A}_{\lambda_5}$ inside $\widetilde{\mathsf{I}}^{5}$. Compare description in Figure~\ref{fig:laChain}. }
    \label{fig:la5chai}
\end{figure}
\end{example}

It must be noted that the uncountable set $\mathcal{T}$ of Theorem \ref{thm:uncountLaZ} contains $\lambda_5\approx-0.366+0.520i$. In fact, such set was found by perturbing the number of repeating coefficients in the power series $f$ and by allowing some of the perturbed repeating coefficients to be $0$. The method Bandt and Hung used to prove that no other power series $g\in\mathcal{P}$ has a root in $\mathcal{T}$, entailed finding a uniform lower bound in a neighborhood of $\mathcal{T}$ on the normalized difference 
\[\frac{g(z)-h(z)}{z^k}=\sum_{j=0}^\infty \epsilon_jz^j,\quad \epsilon_j\in\{-2,-1,0,+1,+2\};~ \epsilon_0\neq0\]
where $k\geq1$ and $h(z)$ is a power series with a root in $\mathcal{T}$. It is unclear whether all parameters in $\mathcal{T}$ are accessible, since the associated itineraries are not necessarily preperiodic and hence our theorems do not apply.

\begin{remark}
We want to emphasize that the assumptions of Theorem~\ref{thm:bdholem} are not meant to be necessary for accessibility. The landmark point $\lambda_6\approx0.57395 + 0.368989 {\rm i}$ satisfies $\mathcal{F}_\lambda=\{1-z+z^3/(1-z)\}$. Based on computer pictures, Bandt \cite{Ba} describes $\lambda_6$ as the ``tip of the largest (visible) spiral of $\mathcal{M}$''. It is likely that $\lambda_6$ is accessible, but even though it is possible to prove that the disks of the chain exist, they are all disconnected and intersect the instar at every level. 
\end{remark}

\section*{Appendix}
We are grateful to the anonymous referees for their comments, in particular for the suggestion to add a table of notation. We include that below.

For completeness, the table makes mention of Hutchinson's operators on $\mathcal{K}$, the space of compact sets in $\mathbb{C}$. The definitions of attractors and instars in terms of these operators are completely analogous to the approach adopted in this text.

Note further that all entries in the table assume a fixed choice of the parameter $\lambda$, except of course the last row, where parameter spaces are defined. 

{\renewcommand{\arraystretch}{1.4}%
\begin{tabular}{|l|c|}
\hline
{\bf Generating maps} &
$\mathfrak{s}_-(z):=-1+\lambda z,\, \mathfrak{s}_{\mathtt{O}}(z):=\lambda z,\, \mathfrak{s}_+(z):=1+\lambda z$ \\
\hline
{\bf Hutchinson operators} &
$S(A) := \mathfrak{s}_-(A) \cup \mathfrak{s}_+(A)$ \\
$S,\widetilde{S}:\mathcal{K} \longrightarrow \mathcal{K}$ &
$\widetilde{S}(A) := \mathfrak{s}_-(A) \cup \mathfrak{s}_0(A)  \cup \mathfrak{s}_+(A)$ \\
\hline
{\bf $2$-attractor} &
 $\mathsf{A}_\lambda:=\left\{\sum_{n=0}^\infty a_n\lambda^n ~\bigg|~a_n\in\{-1,+1\}\right\}$ \\
(\text{fixed point of } $S$) &
$S(\mathsf{A}_{\lambda}) = \mathsf{A}_{\lambda}$ \\
\hline
{\bf $3$-attractor} &
$\widetilde{\mathsf{A}}_\lambda:=\left\{\sum_{n=0}^\infty a_n\lambda^n ~\bigg|~a_n\in\{-1,0,+1\}\right\}$ \\
(\text{fixed point of } $\widetilde{S}$) &
$\widetilde{S}(\widetilde{\mathsf{A}}_{\lambda}) = \widetilde{\mathsf{A}}_{\lambda}$ \\
\hline
{\bf Instars} &
$\mathsf{I}^n := \mathfrak{s}_-(\mathsf{I}^{n-1})\cup\mathfrak{s}_+(\mathsf{I}^{n-1})$ \\
($\mathsf{I}^{-1}, \widetilde{\mathsf{I}}^{-1} := D_R$, where $D_R$ is a &
$ = S^n(\mathsf{I}^{-1})$ \\
disk of radius $R\geq(1-\left|\lambda\right|)^{-1})$ &
$\widetilde{\mathsf{I}}^n := \mathfrak{s}_-(\widetilde{\mathsf{I}}^{n-1})\cup\mathfrak{s}_{\mathtt{O}}(\widetilde{\mathsf{I}}^{n-1})\cup\mathfrak{s}_+(\widetilde{\mathsf{I}}^{n-1})$ \\
 &
$ = \widetilde{S}^n(\widetilde{\mathsf{I}}^{-1})$ \\
\hline
{\bf Overlap set} &
$O_\lambda:=\mathsf{A}_\lambda^-\cap \mathsf{A}_\lambda^+$ \\
\hline
{\bf Power series spaces} &
$\mathcal{P} := \left\{f(z)=\sum_{j=0}^\infty c_jz^j~\bigg|~c_j\in\{-1,0,+1\},~c_0=1\right\}$ \\
 &
$\mathcal{F}_\lambda := \big\{f\in\mathcal{P}~|~ f(\lambda)=0\big\}$ \\
\hline
{\bf Symbolic dynamics} &
$\Sigma^n = \text{ words } \mathtt{\bm w}=a_0\cdots a_{n-1} \text{ in the alphabet } \{-,+\}$ \\
 &
$\Sigma^{\infty} = \text{ infinite words } \mathtt{\bm w'}=a_0\cdots a_{n-1}\cdots$ \\
 &$\mathfrak{s}_{\mathtt{\bm w}}=\mathfrak{s}_{a_0}\circ\ldots\circ\mathfrak{s}_{a_{n-1}}$ \\
 &
$\mathsf{A}_\lambda^\mathtt{\bm w}=\mathfrak{s}_{\mathtt{\bm w}}(\mathsf{A}_\lambda)${\ } (\text{so that }$\mathsf{A}_\lambda=\bigcup_{\mathtt{\bm w}\in\Sigma^n}\mathsf{A}_\lambda^\mathtt{\bm w}$) \\
Projection $\pi_\lambda:\Sigma^\infty\longrightarrow\mathsf{A}_\lambda$ &
$\pi_\lambda(\mathtt{\bm w'}) := \sum_{n=0}^\infty a_n\lambda^n$ \\
 &
($\mathtt{\bm w'}\in\Sigma^\infty$ is the {\em itinerary} of $\pi_\lambda(\mathtt{\bm w'})$ in $\mathsf{A}_\lambda$) \\
\hline
\hline
{\bf Parameter spaces} &
$\mathcal{M}:=\{\lambda\in\mathbb{D}~|~\mathsf{A}_\lambda \mbox{ is connected} \}$ \\
 &
$ = \big\{\lambda\in\mathbb{D}~|~0\in\mathfrak{s}_-(\widetilde{\mathsf{A}}_\lambda)\cap\mathfrak{s}_{\mathtt{O}}(\widetilde{\mathsf{A}}_\lambda)\cap\mathfrak{s}_+(\widetilde{\mathsf{A}}_\lambda)\big\}$ \\
 &
$\mathcal{M}_0:=\{\lambda\in\mathbb{D}~|~0\in\mathsf{A}_\lambda \}$ \\
 &
$ = \big\{\lambda\in\mathbb{D}~|~0\in\mathfrak{s}_-(\mathsf{A}_\lambda)\cap\mathfrak{s}_+(\mathsf{A}_\lambda)\big\}$ \\
\hline
\end{tabular} }

\newpage

\bibliography{SilvestriPerez19_AccessibilityBddThurstonSet}{}
\bibliographystyle{alpha}
\end{document}